\documentclass[leqno,12pt]{amsart}
\usepackage[backref]{hyperref}

\usepackage[textwidth=125mm, textheight=195mm]{geometry}

\usepackage{graphicx}

\usepackage{amsmath,amsthm,amsfonts,amssymb,latexsym,amscd,enumerate, fullpage}
\usepackage[all]{xy}
\usepackage{xy}
\usepackage{xcolor}
\xyoption{all}
\usepackage{boxedminipage}
\usepackage{tikz}
\usepackage{tikz-cd}
\usetikzlibrary{arrows}
\usepackage{mathrsfs}
\usepackage{chngcntr}
\usepackage{apptools}

\newenvironment{psmallmatrix}
{\left(\begin{smallmatrix}}
	{\end{smallmatrix}\right)}

\theoremstyle{plain}
\newtheorem{theorem}{Theorem}
\newtheorem{lemma}[theorem]{Lemma}
\newtheorem{corollary}[theorem]{Corollary}
\newtheorem{proposition}[theorem]{Proposition}
\newtheorem{conjecture}[theorem]{Conjecture}

\theoremstyle{definition}
\newtheorem{definition}[theorem]{Definition}

\newtheorem{example}[theorem]{Example}

\newtheorem{remark}[theorem]{Remark}

\theoremstyle{remark}

\numberwithin{equation}{section}

\newcommand{\R}{\mathbb{R}}
\newcommand{\C}{\mathbb{C}} 
\newcommand{\N}{\mathbb{N}}
\newcommand{\Z}{\mathbb{Z}} 
\newcommand{\Q}{\mathbb{Q}}
\newcommand{\kp}{\mathfrak{p}} 

\newcommand{\cE}{\mathcal{E}}

\newcommand{\cL}{\mathcal{L}}
\newcommand{\cH}{\mathcal{H}}

\newcommand{\cK}{\mathcal{K}}

\newcommand{\cR}{\mathcal{R}}

\DeclareMathOperator{\ch}{ch}

\DeclareMathOperator{\SO}{SO}

\DeclareMathOperator{\Spin}{Spin}

\newcommand{\Spinc}{\Spin^c}
\newcommand{\dirac}{\partial \!\!\! \slash}

\DeclareMathOperator{\Ad}{Ad}

\DeclareMathOperator{\id}{id}

\DeclareMathOperator{\DInd}{D-Ind}
\DeclareMathOperator{\KInd}{K-Ind}

\DeclareMathOperator{\L2ind}{\text{$L^2$}-index}
\DeclareMathOperator{\ind}{index}

\newcommand{\Aav}{\hat A_G}

\begin{document}

	\title[Positive Scalar Curvature and Poincar\'{e} Duality for Proper Actions]
	{Positive Scalar Curvature and Poincar\'{e} Duality\\ for Proper Actions}
	

	\author{Hao Guo}
	\address{School of Mathematical Sciences, University of Adelaide, Adelaide SA 5005 Australia. }
	\email{oug.oah@gmail.com}
	
	\author{Varghese Mathai}
	\address{School of Mathematical Sciences, University of Adelaide, Adelaide SA 5005 Australia. }
	\email{mathai.varghese@adelaide.edu.au}
	
	\author{Hang Wang}
	\address{School of Mathematical Sciences, University of Adelaide, Adelaide SA 5005 Australia. }
	\email{hang.wang01@adelaide.edu.au}
	
	\thanks
	{\noindent V.M. is partially supported by the Australian Research Council via ARC Discovery Project grants DP170101054 and DP150100008.
		H.W. is supported by the Australian Research Council via ARC DECRA DE160100525, and thanks Yoshiyasu Fukumoto for discussions on geometric $K$-homology. H.G. is supported by a University of Adelaide Divisional Scholarship. }
	
	

	\begin{abstract}
		For $G$ an almost-connected Lie group, we study $G$-equivariant index theory for proper co-compact actions with various applications, including obstructions to and existence of $G$-invariant Riemannian metrics of positive scalar curvature. We prove a rigidity result for almost-complex manifolds, generalising Hattori's results, and an analogue of Petrie's conjecture. When $G$ is an almost-connected Lie group or a discrete group, we establish Poincar\'e duality between 
		$G$-equivariant $K$-homology and $K$-theory, observing that Poincar\'e duality does not necessarily hold for general $G$.

	\end{abstract}
	
	\maketitle
	
	\section{Introduction}
	
	In this paper, we study $G$-equivariant index theory for proper co-compact actions with various applications, including obstructions to $G$-invariant Riemannian metrics of positive scalar curvature. As a corollary, we establish an alternate short proof of a recent result of Weiping Zhang \cite{Zhang17} in Section 6. We also prove the existence of $G$-invariant Riemannian metrics of positive scalar curvature under certain very general hypotheses on $G$ and its action on the manifold (see Theorem \ref{thm:PSCexistence}).
	
	Our $G$-equivariant index theory is applied to prove rigidity theorems in Section 7.1  for certain $\Spinc$-Dirac operators on almost-complex manifolds with a proper $G$-action,
	generalising results of Hattori \cite{Hat} (see also Atiyah-Hirzebruch \cite{AH} and Hochs-Mathai \cite{HM16}). In addition, we prove an analogue of Petrie's conjecture \cite{Petrie}
	in Section 7.2.

	
	We begin in Section \ref{sec:equivalence} by defining the $G$-equivariant geometric $K$-homology $K^{\textnormal{{geo,}} G}_{\bullet}(X)$ for almost-connected $G$ as a natural extension of the original definition of Baum and Douglas \cite{Baum-Douglas82} and the definition in the compact group case \cite{BOSW}. It is a description of $K$-homology as a quotient of the equivariant bordism group over $X$. We extend the result in \cite{BOSW} for compact groups to show that it is isomorphic, via the Baum-Douglas map, to the analytic $G$-equivariant $K$-homology $K^G_{\bullet}(X)$ defined by Kasparov \cite{Kasparov, Kasparov.K-homology}.
	
	
	Part of the technique we develop there, namely induction on analytic $K$-homology, is then used in Section \ref{sec:PD} to prove one of our main theorems: equivariant Poincar\'e duality for almost-connected Lie groups and discrete groups. In the almost-connected case, this comes in the form of an isomorphism
	\begin{equation}\label{eqn:PD}
	\mathcal{P}\mathcal{D}: K^G_{\bullet}(C_\tau(X))\simeq K^G_{\bullet}(X),
	\end{equation}
	where $C_\tau(X)$ is an algebra of continuous sections of the Clifford bundle. Our result extends equivariant Poincar\'{e} duality for compact groups, which is a consequence of work by Kasparov \cite{Kasparov}. Our proof uses a result of Phillips \cite{Phillips2} showing that there is an ``induction" isomorphism of $K$-theory groups $K^{K}_{\bullet}(C_\tau(Y)) \simeq K^{G}_{\bullet+d}(C_\tau(X))$, where $K$ is a maximal compact subgroup of $G$ and $Y$ is a $K$-slice of $X$. We also establish in this section versions of Poincar\'e duality when $X$ is either not $G$-cocompact or has boundary, and relate our results to those obtained by Kasparov in \cite{Kasparov}. In Section 3.2, we make use of previous work in \cite{LO} and \cite{BaumHigsonSchick.eq} on equivariant $K$-theory and $K$-homology for discrete groups to prove equivariant Poincar\'{e} duality for discrete groups using the Mayer-Vietoris sequence.
	
	The significance of our results on Poincar\'{e} duality can be seen as two-fold. First, they establish that equivariant Poincar\'{e} duality holds for a large class of {\em non-compact} Lie groups, and is an extension to non-compact groups the type of Poincar\'{e} duality proved in \cite{Kasparov}. Second, the fact that the duality holds for these groups places into context the observation, made by Phillips in \cite{Phillips1} and L\"uck-Oliver in \cite{LO} \S5, that for an {\em arbitrary} Lie group $G$, an equivariant $K$-theory constructed from finite-dimensional $G$-vector bundles may not always be a generalised cohomology theory. This means one cannot expect Poincar\'{e} duality - as formulated in this way (for another way to formulate it in the case of non-compact group actions see Emerson and Meyer \cite{EmersonMeyer2} and \cite{EmersonMeyer1}) - to hold when $G$ is an arbitrary Lie group. A concrete non-example of a group for which it does not hold is the semi-direct product $G=(S^1\times S^1) \rtimes_\alpha \Z$, where $\alpha (n)= \begin{psmallmatrix}1 & 0\\n & 1\end{psmallmatrix},\, n\in \Z$, which is a non-linear Lie group. As has been observed by others, in cases such as this one, Phillips' generalisation of equivariant $K$-theory \cite{Phillips1} is not the same as the definition via $C^*$-algebras. 
	
	We observe that one general situation in which they are equivalent is when $G$ has only finitely many connected components (see Lemma \ref{lem:phil}), and this fact guides our proof of Poincar\'{e} duality in the almost-connected group case.
	
	In Section \ref{sec:Index theory of a K-homology class} we study the relationship of $K$-equivariant index theory to $G$-equivariant index theory, where $K$ is again a maximal compact subgroup of $G$ (compare \cite{MZ}). The difference between our approach and, for instance, the approach in \cite{HochsDS}, is that we present a more direct proof, as well as give applications, of the fact that the following diagram relating the compact and non-compact indices commutes:
	\begin{equation}
	\label{indexcommutes}
	\begin{CD}
	K_{\bullet}^K(Y) @>\ind_K>> K_{\bullet}(C^*_r(K))\\
	@V\simeq V\KInd V @V\simeq V\DInd V\\
	K_{\bullet+d}^G(X) @>\ind_G>>K_{\bullet+d}(C^*_r(G)).
	\end{CD}
	\end{equation}
	Here, the left vertical arrow is analytic induction from $K$ to $G$, and the 
	right vertical arrow is Dirac induction from $K$ to $G$.
	
	One of the consequences of this result (see Section \ref{sec:reduction}) is an elegant integral trace formula for the special case when $G$ is a connected, semisimple Lie group with finite centre, $\dim G/K$ is even and $K$ has maximal rank. This result builds on the work of Atiyah-Schmid \cite{AtiyahSchmid} and Lafforgue \cite{LafforgueICM}, but combines it with the integral formula for an $L^2$-index, proved by Wang in \cite{Wang}. More precisely, we prove that if $M$ is a $G$-Spin$^c$-manifold and $N$ a $K$-slice, then the previous commutative diagram  Then the diagram \ref{indexcommutes} fits into a larger commutative diagram involving the von Neumann trace on the $K$-theory of $C_r^*(G)$, denoted by $\tau_G$, and a formula for the formal degree of discrete series representations in terms of the root systems of $K$ and $G$ (see \cite{LafforgueICM}):
	\begin{equation*}\label{tracediagram}
	\begin{tikzcd}
	K_0^G(M)\arrow{r}{\ind_G} & K_0(C_r^*(G))\arrow{rd}{\tau_G} & \\ 
	& & \mathbb{R}.\\
	K_0^K(N)\arrow{uu}{\KInd}\arrow{r}{\ind_K} & R(K)\arrow{uu}{\DInd}\arrow{ru}{\Pi_K} &
	\end{tikzcd}
	\end{equation*}
	Here the quantity $\Pi_K([V_\mu]) :=\prod_{\alpha\in\Phi^+}\frac{(\mu+\rho_c,\alpha)}{(\rho,\alpha)}$ (for details of the notation see Section \ref{sec:reduction}). Using the result of Wang \cite{Wang}, we obtain naturally an equality of two integrals of characteristic classes, one on $M$ and the other on the compact slice $N$.
	
	
	\vspace{0.3cm}
	
	\tableofcontents

	\section{Equivalence of Analytic and Geometric $K$-homologies} \label{sec:equivalence}

	The goal of this section is to prove Theorem~\ref{thm:eq.K-homology}, which states the equivalence of two $K$-homology theories for a large class of groups and spaces, building on the work done in \cite{BaumHigsonSchick}, \cite{BaumHigsonSchick.eq} and \cite{BOSW}. Let us give a very brief introduction to these two theories.
	
	{\em Analytic $K$-homology} was first studied by Atiyah~\cite{Atiyah70}, who was motivated by the classification of elliptic pseudo-differential operators on a locally compact topological space. Brown-Douglas-Fillmore~\cite{BDF} then understood it for $C^*$-algebras from the point of view of extension, before Kasparov~\cite{Kasparov.K-homology} formulated it in the general setting of $KK$-theory. It is defined using analytic $K$-homology cycles.
	
	\begin{definition}[{\cite{Kasparov.K-homology, BCH}}]
		\label{def:ana.K.homology}
		Let $G$ be a locally compact group acting properly on a Hausdorff space $X$. 
		An \emph{analytic $K$-homology cycle}, or \emph{Kasparov cycle}, is a triple of the form $(\cH, \phi, F)$, where
		\begin{itemize}
			\item $\cH$ is a $\Z_2$-graded $G$-Hilbert space, 
			\item $\phi: C_0(X)\rightarrow\cL(\cH)$ is an even $G$-equivariant $*$-homomorphism and 
			\item $F$ is an odd self-adjoint bounded linear operator on $\cH$, 
		\end{itemize}
		where $\mathcal{L}(\cH)$ is the $C^*$-algebra of bounded operators on $\cH$, such that the elements 
		\[
		\phi(a)(F^2-1), [\phi(a), F], [g, F]
		\] 
		belong to the $*$-subalgebra $\cK(\cH)$ consisting of the compact operators, for any $a\in C_0(X)$ and any $g\in G.$ 
		The (even) \emph{$G$-equivariant analytic $K$-homology of $X$}, denoted $K^G_0(X)$, is the abelian group generated by Kasparov cycles subject to certain equivalence relations given by homotopy. 
		There is also an odd part, $K^G_1(X)$, whose elements are represented by cycles $(\cH, \phi, F)$ with no imposed $\Z_2$-grading.
	\end{definition}
	
	\begin{remark}
		When $A$ and $B$ are $G$-$C^*$-algebras, a Kasparov $(A, B)$-cycle is defined similarly as in Definition~\ref{def:ana.K.homology}, with the difference that $\cH$ is now a Hilbert $B$-module and $C_0(X)$ is replaced by $A$ (see \cite{Kasparov.K-homology} Definitions 2.2 and 2.3). Equivalence classes of such cycles form an abelian group $KK^G(A, B) := KK_0^G(A,B)$. There also exists an odd part $KK^G_1(A,B)$. $KK$-theory encompasses analytic $K$-homology in the sense that
		\[
		KK_{\bullet}^G(C_0(X), \C)\simeq K^G_{\bullet}(X).
		\]
	\end{remark}
	
	On the other hand, geometric $K$-homology was introduced by Baum and Douglas~\cite{Baum-Douglas82} from the perspective of Spin geometry and Dirac operators. 
	
	\begin{definition}[\cite{Baum-Douglas82, BaumHigsonSchick}]
		\label{def:geoKhomo}
		Let $X$ be a $G$-space. 
		A \emph{geometric $K$-homology cycle} is a triple of the form $(M, E, f)$, 
		where 
		\begin{itemize}
			\item $M$ is a proper $G$-cocompact manifold with a $G$-equivariant $\Spinc$-structure, 
			\item $E$ is a smooth Hermitian $G$-equivariant vector bundle over $M$ and 
			\item $f: M\rightarrow X$ is a continuous $G$-equivariant map.
		\end{itemize}  
		The $G$-\emph{equivariant geometric $K$-homology}, which we shall write as $K_{\bullet}^{\textnormal{{geo,}} G}(X)$, where $\bullet = 0$ or $1$, is an abelian group generated by geometric cycles $(M,E,f)$ where $\dim M \equiv 1$ or $0$ mod 2, subject to an equivalence relation generated by three operations: {\em direct sum/disjoint union}, {\em bordism} and {\em vector bundle modification}. The first relation amounts to the identification
		\[
		(M\sqcup M, E_1\sqcup E_2, f\sqcup f)\sim(M, E_1\oplus E_2, f),
		\]
		while the statement of the latter two operations are more involved. For their definitions we refer to pp. 5-6 of \cite{BaumHigsonSchick.eq}. We will follow the definitions contained there, except that $G$ for us is an almost-connected Lie group instead of a discrete group. 
		
		Addition in the group is given by 
	 	$$(M_1,E_1,f_1) + (M_2,E_2,f_2) = (M_1\sqcup M_2,E_1\sqcup E_2,f_1\sqcup f_2),$$ with the additive inverse of $(M,E,f)$ being given by the cycle $(-M,E,f)$, where by $-M$ we mean the $G$-$\Spinc$-manifold with the opposite $G$-$\Spinc$-structure to $M$ (see p.5 of \cite{BaumHigsonSchick.eq}).	 	
	\end{definition}
	Implicit in the definition of a bordism $W$ from a $G$-manifold $M_1$ to a $G$-manifold $M_2$ is that $M_1$ and $M_2$ are contained in $G$-equivariant collar neighbourhoods at the ends of $W$. The existence of such neighbourhoods for proper actions is proved in \cite{Kankaanrinta} Theorem 3.5.
	\begin{example}
		\label{ex:fundamental.class}
		Let $K$ be a compact Lie group and $Y$ a compact $K$-manifold. If $Y$ has a $K$-equivariant $\Spinc$-structure, we can define the \emph{fundamental class} of $K^{\textnormal{{geo,}} K}_{n}(Y)$ to be
		\[
		[Y]_K:=[(Y, Y\times \C,\id_Y)]\in K^{\textnormal{{geo,}} K}_{n}(Y),\qquad n=\dim Y \text{ (mod}~2).
		\]
		Similarly, if $G$ is an almost-connected Lie group acting properly on a $G$-cocompact manifold $X$ with a $G$-invariant $\Spinc$-structure, the fundamental class of $K^{\textnormal{{geo,}} G}_{n}(X)$ is given by 
		\[
		[X]_G:=[(X, X\times \C,\id_X)]\in K^{\textnormal{{geo,}} G}_{n}(X),\qquad n=\dim X \text{ (mod}~2).
		\]
		One of the themes of this section of the paper is that when $K$ is a maximal compact subgroup of $G$ and $Y$ sits suitably inside $X$ as a $K$-invariant submanifold, there is a one-to-one correspondence between such classes, provided the homogeneous manifold $G/K$ is also $G$-$\Spinc$.
	\end{example}
	
		Baum and Douglas introduced, originally in the non-equivariant setting, a natural map from geometric $K$-homology to analytic $K$-homology \cite{Baum-Douglas82}; it extends naturally to the equivariant setting. Our version of this map, which is to say in the setting when $G$ is an almost-connected Lie group, will be denoted by
	\begin{equation}
	\label{eq:BaumDoug}
	BD: K^{\textnormal{{geo,}} G}_{\bullet}(X)\rightarrow K^G_{\bullet}(X),
	\end{equation}
	and defined as follows. Let $(M,E,f)$ be a geometric $K$-homology cycle for $X$. Since $M$ is $G$-equivariantly $\Spinc$ by assumption, one can construct a {\em $\Spinc$-Dirac operator} on $M$, which acts on a  spinor bundle $S_M\rightarrow M$. For reference in later sections, note that $S_M$ is locally constructed by tensoring a Hermitian connection on the determinant line bundle $L$ (associated to the given $\Spinc$-structure) with the lift of the Levi-Civita connection on $TM$ to the local spinor bundle. Note also that if $M$ is Spin, the Spin$^c$-Dirac operator can be realised as the Spin-Dirac operator twisted by a line bundle. 
	
	Let $D_E$ be the operator $D$ twisted by the vector bundle $E$. It can be viewed as an unbounded, densely-defined operator on the Hilbert space $L^2(S_M\otimes E)$. Let $m$ be the representation of $C_0(M)$ on $L^2(S_M\otimes E)$ given by pointwise multiplication of functions on sections. Since $D_E$ is self-adjoint, we can use functional calculus to form the $L^2$-bounded operator $D_E(1+D_E^2)^{-\frac{ 1}{2}}$. Then the triple
	$$\left(L^2(S_M\otimes E), m, D_E(1+D_E^2)^{-\frac12}\right)$$
	is an analytic $K$-homology cycle defining a class in $K_0^G(M)$. Call this class $[D_E]$ (we may also denote it as a cap product $[E]\cap[D]$, depending on the context). Let $f': C_0(X)\rightarrow C_0(M)$ be the contravariant map on algebras given by 
	\begin{equation}
	\label{eq:f'f}
	(f'(g))(m)=g(f(m)),\qquad g\in C_0(X),\, m\in M.
	\end{equation}
	Then the image of $[(M,E,f)]$ under $BD$ is defined to be the class in $K^G_0(X)$ represented by the Kasparov cycle
	\[
	\left(L^2(S_M\otimes E), m\circ f', D_E(1+D_E^2)^{-\frac12}\right).
	\]
	We will also write this class using the pushforward notation $f_*([D_E])$. Notice that $$f_*([D_E])=f'^*([D_E]),$$ where $f'^*:K^G_0(M)\rightarrow K^G_0(X)$ is the covariant functorial map on analytic $K$-homology.

	It has long been conjectured that equivariant analytic and geometric $K$-homologies are in general equivalent. This was proved in the case that $X$ is a compact CW-complex without group action in \cite{BaumHigsonSchick}, some twenty-five years after the conjecture was posed. Subsequently, the cases of cocompact discrete group action~\cite{BaumHigsonSchick.eq} and compact Lie group acting on a compact CW-complex~\cite{BOSW} were confirmed. Our main aim in this section is to confirm this conjecture for the case of $X$ a manifold admitting a proper cocompact action of an almost-connected Lie group $G$. More precisely, we prove:
	
	\begin{theorem}
		\label{thm:eq.K-homology}
		Let $G$ be an almost-connected Lie group acting smoothly, properly and cocompactly on a manifold $X$. If $G/K$ admits a $G$-equivariant $\Spinc$-structure, the Baum-Douglas map relating $G$-equivariant analytic and geometric $K$-homologies is an isomorphism
		\begin{equation}
		\label{eq:eq.K-homology}
		K^{\textnormal{{geo,}} G}_{\bullet}(X)\simeq K^G_{\bullet}(X).
		\end{equation}
	\end{theorem}
	
	\subsection{Overview of the Proof}
	
	We prove Theorem~\ref{thm:eq.K-homology} in several steps. Relevant to us will will be Abels' global slice theorem, which we recall presently.
	
	\begin{theorem}[Abels \cite{Abels}] 
		\label{thm:abels}
		Let $G$ be an almost-connected Lie group and $K$ a maximal compact subgroup of $G$. Then $X$ has a global $K$-slice, defined by
		\[
		Y=f^{-1}(eK) \subset X,
		\] 
		where $f: X\rightarrow G/K$ is a $G$-equivariant smooth map. 
		$Y$ is a $K$-invariant submanifold and $X$ is diffeomorphic to the associated space 
		\begin{equation}
		\label{eq Abels fibr}
		G\times_K Y:= G\times Y/\{(gh, y)\sim(g, hy), \forall h\in K\}.
		\end{equation}
	\end{theorem}
	
	The associated space \eqref{eq Abels fibr} is a fibre bundle over $G/K$ with fibre the manifold $Y.$ The $G$-equivariant diffeomorphism $G\times_K Y\rightarrow X$ is given by $[(g, y)]\mapsto g\cdot y$, where $[(g, y)]$ denotes the equivalence class of the pair $(g, y)\in G\times Y$ in the quotient $G\times_K Y.$
	
	\begin{remark}\label{rem:Slice.unique}	
		In the proofs that follow, we will make frequent and essential use of the fact that the slice $Y$ in Abels' theorem is {\em essentially unique} in a rather strong sense, namely the $K$-diffeomorphism class of the slice $Y$ depends only upon the $G$-diffeomorphism class of the $G$-action on $X$ (for the proof  of this fact see Theorem 2.2 in the paper of Abels \cite{Abels} and the first paragraph of page containing it). In particular, this means that for a fixed $G$-action on $X$, the slice $Y$ is unique up to $K$-diffeomorphism.
	\end{remark}
	
	Let $d := \dim G/K$. In Section~\ref{sec:Induction on geometric K-homology}, we show carefully that that Abels' theorem can be used to write down a well-defined ``induction map" on equivariant geometric $K$-homology,
	\[
	i: K^{\textnormal{{geo,}} K}_{\bullet}(Y)\rightarrow K^{\textnormal{{geo,}} G}_{\bullet+d}(X),
	\]
	and show it is an isomorphism by slightly extending Abels' theorem. In Section~\ref{sec:Induction on analytic K-homology} we define a corresponding induction map on analytic $K$-homology, which we prove is an isomorphism using tools from $KK$-theory.
	\[
	j: K^K_{\bullet}(Y)\rightarrow K^G_{\bullet+d}(X).
	\]
	Finally, we recall the following key theorem of Baum, Oyono-Oyono, Schick and Walter, which shows that equivariant geometric and analytic $K$-homologies are equivalent in the case of compact group actions (note that any smooth manifold admits a CW-complex structure):
	
	\begin{theorem}[\cite{BOSW}]
		\label{thm:BO2SW}
		Let $Y$ be a compact CW-complex with an action of a compact Lie group $K$. Then the Baum-Douglas map is an isomorphism
		\[
		K^{\textnormal{{geo,}} K}_{\bullet}(Y)\simeq K^K_{\bullet}(Y).
		\]
	\end{theorem}
	
	This isomorphism and the isomorphisms $i$ and $j$ fit into the following diagram:
	\begin{equation}
	\label{eq:comm.K.homo}
	\begin{CD}
	K^{\textnormal{{geo,}} K}_{\bullet}(Y) @>BD>> K^K_{\bullet}(Y)\\
	@ViVV @VjVV\\
	K^{\textnormal{{geo,}} G}_{\bullet+d}(X) @>BD>> K^G_{\bullet+d}(X),
	\end{CD}
	\end{equation} 
	where the bottom arrow is the $G$-equivariant Baum-Douglas map we defined previously. The final step is to show that this diagram commutes (Section~\ref{sec:Commutativity of the diagram}). From this Theorem~\ref{thm:eq.K-homology} follows.
	
	\begin{remark}
	The assumption that $G/K$ is $G$-$\Spinc$ is only used in proving Proposition~\ref{lem:iso.geo.ind}.
	\end{remark}

	\subsection{Induction on Geometric $K$-homology}
	\label{sec:Induction on geometric K-homology}
	
	We define an induction map on equivariant geometric $K$-homology and show that, when $G/K$ has a $G$-invariant $\Spinc$-structure, it is an isomorphism. Before doing so, we need to make an important remark about the relationship between $\Spinc$-structures on a $G$-manifold $M$ and a given $K$-slice $N$.
	\begin{remark}
		\label{rem:Spinc.stru.ind}
		In \cite{HochsMathaiAJM} Section 2.3, in particular Definition 2.7, a procedure called {\it stabilisation} is given to define a $G$-equivariant $\Spinc$-structure starting from a $K$-equivariant $\Spinc$-structure on $N$, together with a procedure called {\it destabilisation} going the other way, both based on Plymen's two-out-of-three lemma. Further, Lemma 3.9 of the same paper proves that stabilisation and destabilisation are inverses of one another. Thus there is a one-to-one correspondence between the collection of $G$-$\Spinc$-structures on $M$ and the collection of $K$-$\Spinc$-structures on $N$. We note that a similar correspondence holds when $M$ and $N$ have boundary (see also. In the rest of this paper, whenever we speak of a ``corresponding" or ``compatible" $G$-$\Spinc$-structure or $K$-$\Spinc$ structure given the other, we shall be doing so with this correspondence in mind.
	\end{remark}
	\begin{definition}\label{def:geometricinduction}
	Let $G, X, K$ and $Y$ be as before. The {\em geometric induction map} is given by
	$$i: K^{\textnormal{{geo,}} K}_{\bullet}(Y)\rightarrow K^{\textnormal{{geo,}} G}_{\bullet+d}(X),$$
	$$ [(N, E, f)]\mapsto [(G\times_K N, G\times_K E, \tilde f)]=:[(M, \tilde E, \tilde f)].$$
	Here $G\times_K N$ is equipped with the corresponding $\Spinc$-structure (in the sense of the remark above), while the map $f: M = G\times_K N\rightarrow G\times_K Y \cong X$ is the natural $G$-equivariant map determined by the $K$-equivariant map $f: N\rightarrow Y$ on the fibre. The vector bundle $G\times_K E$ can be defined as the pullback bundle $pr_2^*E\rightarrow G\times Y$ of $E$ along the projection $pr_2:G\times Y\rightarrow Y$ modulo the $K$-action on $G\times E$ given by $k(g,e) = (gk^{-1},ke)$, and equipped with the pull-back Hermitian metric.
	\end{definition}
	\begin{proposition}
	The map $i$ is a well-defined homomorphism of abelian groups.
	\end{proposition}
	\begin{proof}
	One sees immediately that $(M, \tilde E, \tilde f)$ satisfies the requirements of a geometric cycle. It is also easy to verify that the map $i$ is well-defined with respect to the disjoint union/direct sum relation. To see that $i$ is well-defined with respect to bordism, suppose that $(N_1,E_1,f_1)$ and $(N_2,E_2,f_2)$ are bordant cycles via a triple $(S,E,f)$. The main point is that $G\times_K N_1$ and $-G\times_K N_2$ are bordant in the following natural way. By hypothesis, we have $\partial S = N_1\sqcup -N_2$. Since $S$ is a $K$-manifold with boundary, and $K$ acts by self-diffeomorphisms of $S$, the $K$-action separates into two parts: an action on the boundary and an action on the interior of $S$. Therefore $G\times_K S$ is a $G$-manifold with boundary $G\times_K\partial S$ and is equal to the disjoint union of $G\times_K N_1$ and $-G\times_K N_2$. The latter two manifolds are bordant via $G\times_K S$. One then verifies that the vector bundle $G\times_K E$ restricts to $G\times_K E_1$ and $G\times_K E_2$ at the ends of the bordism, and similarly that $\tilde{f}$ restricts to $\tilde{f}_1$ and $\tilde{f}_2$.
	
	Next, suppose that $(N_1,E_1,f_1)$ and $(N_2,E_2,f_2)$ are related by a vector bundle modification. In particular $N_2$ is the total space of the unit sphere sub-bundle of $V\oplus(N_1\times\mathbb{R})\rightarrow N_1$, for some $G$-$\Spinc$-vector bundle $V$ over $N_1$ with fibres of even dimension $2k$. Denote the $G$-$\Spinc$-structure of $V$ by $P_V\rightarrow N_1$. By the definition given on p. 6 of \cite{BaumHigsonSchick.eq}, we have
$$E_2 = (P_V\times_{\Spinc}\beta)\otimes\pi^*(E_1),$$ where $\beta$ is the Bott generator vector bundle over $S^{2k}$ and $\pi:N_2\rightarrow N_1$ is the projection map for the sphere bundle. Observe that $G\times_K N_2$ is just the unit sphere sub-bundle of $G\times_K V\oplus (G\times_K N_2\times\mathbb{R})$, which is still a direct sum of an even-rank vector bundle with the trivial real one-dimensional bundle over $G\times_K N_2$. By inspection, the induced vector bundle $G\times_K E_2 = G\times_K(P_V\times_{\Spinc}\beta)$ over $G\times_K N_2$ is still $\beta$ when restricted over each spherical fibre of $G\times_K N_2$. It follows that the induced cycles $(G\times_K N_1,G\times_K E_1,\tilde{f_1})$ and $(G\times_K N_2,G\times_K E_2,\tilde{f_2})$ are still related by a vector bundle modification, so $i$ is also well-defined with respect to this operation. 

	Finally, since $i$ preserves disjoint unions, it is a homomorphism of abelian groups.
	\end{proof}

	Having seen that it is well-defined, let us show that $i$ is an isomorphism. First we need to establish the following variant of Abels' global slice theorem for manifolds with boundary.
	\begin{lemma}\label{lem:boundaryslice}
	Let $G,K$ be as above, and let $W$ be a proper $G$-manifold with boundary. Then there exists a global $K$-slice $S\subseteq W$ - that is, a $K$-submanifold with boundary such that we have a $G$-equivariant diffeomorphism $W\cong G\times_K S$ with $\partial S = S\cap\partial W$.
	\end{lemma}
	\begin{proof}
	First we make the claim that there exists a $G$-equivariant smooth map $f:W\rightarrow G/K$. We sketch its proof, mainly noting where it differs from the case proved in \cite{Abels} p. 8. Let $\pi:W\rightarrow W/G$ be the natural projection. It can be shown that there exists for each $w\in W$ a $G$-invariant neighbourhood $U_w$ that admits a local $K$-slice. Indeed, if $w$ is in the interior of $W$, this follows from Proposition 2.2.2. in \cite{Palais}, and if $w\in\partial W$ is a boundary point, we can proceed along the lines of Palais in \cite{Palais} Section 2 to obtain $U_w$. Note that the $G$-map $f_{{U}_w}:U_w\rightarrow G/K$ extends to a slightly larger open set containing $\overline{U}_w$, so we get smooth $G$-maps $f_{\overline{U}_w}:\overline{U}_w\rightarrow G/K$. Define the cover
	$$\mathcal{U}:=\{U_w:w\in W\}$$
	of $W$ constructed out of these $U_w$. Since $W/G$ is paracompact, there is an open $\sigma$-discrete refinement of the cover $\pi(\mathcal{U})=\{\pi(U):U\in\mathcal{U}\}$ of $\pi(W)$. That is, there is a sequence $\mathscr{U}_n$, indexed by $n\in\mathbb{N}$, of families of open subsets of $W/G$ such that $\bigcup_{n=1}^\infty(\mathscr{U}_n)$ is a cover of $W/G$ that refines $\pi(\mathcal{U})$, and every family $\mathscr{U}_n$ is discrete (see \cite{Abels} p. 6). 
	
	One can then show, as on p.8 of \cite{Abels}, that the set of restricted $G$-maps $f_{\overline{A}}:=f|_{\overline{A}}:\pi^{-1}(\overline{A})\rightarrow G/K$, with $A$ ranging over the elements of $\mathcal{U}$, piece together to give a composite smooth $G$-map $f:\bigcup_{A\in\mathcal{U}}\pi^{-1}(\overline{A})\rightarrow G/K$. The rest of the construction goes as in \cite{Abels} and yields the desired map $f:W\rightarrow G/K$.

	One observes, by Palais' local slice theorem \cite{Palais} or otherwise, that $eK$ is a regular value for both $f$ and $f|_{\partial W}$. Hence $f^{-1}(eK)=S\subseteq W$ is a submanifold with boundary $\partial S = S\cap\partial W$ (for a proof of this fact see for example \cite{Milnor}). The map $G\times_K S\rightarrow W$ given by $[g,s]\mapsto g\cdot s$ provides a $G$-equivariant diffeomorphism.
	\end{proof}

	\begin{proposition}
		\label{lem:iso.geo.ind}
		Let $G,X,K$ and $Y$ be as before. If the homogeneous manifold $G/K$ has a $G$-equivariant $\Spinc$-structure, the map $i$ is an isomorphism.
	\end{proposition}
	
	\begin{proof}
		Let $(M, E, f)$ be a geometric cycle for $K^{\textnormal{{geo,}} G}_\bullet(X)$, where $\bullet = 0$ or $1$. Note that $M$ is a $G$-manifold with a $G$-equivariant $\Spinc$-structure. 
		By Remark~\ref{rem:Spinc.stru.ind} one can choose a global $K$-slice $N$ with a compatible $K$-equivariant $\Spinc$-structure.  
		The restriction of $E$ and $f$ to the submanifold $N$ is a pre-image $[(N, E|_N, f|_N)]$ of $[(M, E, f)]$.
		Hence $i$ is surjective. 
		
		To show injectivity, let $x_k\in K^{\textnormal{{geo,}} K}_0(Y),\,\, k=1, 2,$ be represented by geometric cycles $(N_k, E_k, f_k)$, such that $i(x_1)=i(x_2)$. That is, we have a relation between cycles
		\begin{equation}
		\label{e1}
		(G\times_K N_1, G\times_K E_1, \widetilde{f_1})\sim (G\times_K N_2, G\times_K E_2, \widetilde{f_2}).
		\end{equation}
		We show that this induces a relation
		\begin{equation}
		\label{e2}
		(N_1, E_1, f_1)\sim (N_2, E_2, f_2).
		\end{equation}
	
		Indeed, if \eqref{e1} is a relation by disjoint union/direct sum, it follows that \eqref{e2} is also, once one remembers that the $K$-slice is essentially unique (see Remark \ref{rem:Slice.unique} or Theorem 2.2 in \cite{Abels}).
			
		Now suppose  \eqref{e1} is a relation by bordism, so that $G\times_K N_1 \sqcup - G\times_K N_2$ is the boundary of another $G$-cocompact $G$-Spin$^c$-manifold $W$, which is part of a triple $(W,\tilde E,\tilde f)$ with $\tilde E$ and $\tilde f$ restricting to $G\times_K E_i$ and $\tilde f_i$ at the ends of $W$. Then by Lemma \ref{lem:boundaryslice}, $W$ has a $K$-slice $S$, with boundary $\partial S = \partial W\cap S$. Since the $G$-action - which by hypothesis is by self-diffeomorphisms of $W$ - preserves the interior and boundary of $W$, one sees that $\partial W$ has a $K$-slice $\partial W_0$ by the usual version of Abels' global slice theorem. Observe that, $G$-equivariantly, 
		$$G\times_K\partial S\cong\partial W\cong G\times_K N_1\sqcup -G\times_K N_2\cong G\times_K(N_1\sqcup -N_2),$$ 
		where the first diffeomorphism follows by dimensional considerations. By uniqueness of $K$-slices up to $K$-diffeomorphism, this means that, $K$-equivariantly, $\partial W_0\cong\partial S\cong N_1\sqcup -N_2$. We now argue that $S$ can be equipped with a $K$-$\Spinc$-structure that is compatible (in the sense of Remark \ref{rem:Spinc.stru.ind}) with the $G$-$\Spinc$-structure on $W$. First note that the boundary pieces $G\times_K N_1$ and $G\times_K N_2$ are contained in $G$-equivariant collar neighbourhoods in $W$ (see the remark above Example \ref{ex:fundamental.class}). This allows one to form the double $\overline{W}$ of $W$, which is then naturally a $G$-$\Spinc$, $G$-cocompact manifold without boundary. Take a $K$-slice $\overline{S}$ of $\overline{W}$, observing that $\overline{S}$ is the double of $S$. By the correspondence mentioned in Remark \ref{rem:Spinc.stru.ind}, $\overline{S}$ has a $K$-$\Spinc$-structure that induces the $G$-$\Spinc$-structure on $\overline{W}$. The manifold $S$ equipped with the restriction of this $K$-$\Spinc$-structure to $S$ is now a $K$-$\Spinc$-bordism from $N_1$ to $N_2$. Hence $N_1$ and $N_2$ are bordant. Next, note that the vector bundle $G\times_K E$ by definition restricts to $G\times_K E_1$ and $G\times_K E_2$ on $\partial W$. Taking the fibre at the identity $eK$ of the fibre bundle $G\times_K E\rightarrow G/K$ then gives a $K$-vector bundle whose restriction to $\partial S$ is isomorphic to $E_1\sqcup E_2$. One sees that, by construction, $\tilde f$ restricts to a $K$-equivariant map $f$ on $S$, such that $f|_{\partial S} = f_1\sqcup f_2$, after identifying $\partial S$ with $N_1\sqcup N_2$. Thus the two reduced geometric cycles $(N_1,E_1,f_1)$ and $(N_2,E_2,f_2)$ are still related by a bordism operation in geometric $K$-homology.

		Finally, if \eqref{e1} is a relation by a vector bundle modification, then there is a $G$-Spin$^c$-vector bundle $V$ over $G\times_K N_1 =: M_1$ with even-dimensional fibres such that $G\times_K N_2 =: M_2$ is the sphere sub-bundle of $(M_1\times\R)\oplus V$, while $G\times_K E_2$ is (modulo a tensor product with $\pi^*(G\times_K E_1)$, where $\pi:M_2\rightarrow M_1$ is the projection of the sphere bundle) the bundle whose fibre over each sphere is the Bott generator vector bundle $\beta$. By Remark \ref{rem:Spinc.stru.ind}, the bundle $V|_{N_1}$ over the slice ${N_1}$ has a compatible $K$-equivariant $\Spinc$-structure. Further, the sphere sub-bundle of $(N_1\times\R)\oplus (V|_{N_1})$ is precisely the restriction of $M_2$ to the submanifold $N_1\subseteq M_1$; it is a $K$-slice of $M_2$. Thus after restriction, the manifolds $N_1$ and $N_2$ in \eqref{e2} are still related by a vector bundle modification. Let us write $G\times_K E_2$ in the form $P_{M_1}\times_{\Spinc}\beta$, where $P_{M_1}$ is the $G$-$\Spinc$-structure on $M_1$. By the correspondence in \cite{HochsMathaiAJM}, $P_{M_1}$ reduces to a compatible $\Spinc$-structure $P_{N_1}$ on $N_1$, and the associated bundle reduces to $P_{N_1}\times_{\Spinc}\beta$, which is still the Bott generator over each spherical fibre of $N_2$. One verifies also that the function $\tilde f\circ\pi$ restricts to $f\circ\check{\pi}$, where $\check{\pi}$ is the projection of the sphere bundle $N_2\rightarrow N_1$. Thus the two geometric cycles $(N_1,E_1,f_1)$ and $(N_2,E_2,f_2)$ are still related by a vector bundle modification.

		Hence the map $i$ is also injective and therefore an isomorphism.
	\end{proof}

	\subsection{Induction on Analytic $K$-homology}
	\label{sec:Induction on analytic K-homology}
	
	As the next step in proving the equivalence of geometric and analytic $K$-homologies, we now turn to the task of applying Abels' global slice theorem in the context of analytic $K$-homology. The goal of this subsection is to first show, using $KK$-theory, that there is have an isomorphism of abelian groups
	\[
	K^K_{\bullet}(Y)\simeq K^G_{\bullet+d}(X),
	\]
	before providing a natural map at the level of analytic $K$-cycles that realises this isomorphism. We recall the following definition (see \cite{Blackadar} for more details). Let $X$ be a $\sigma$-compact $G$-space and $A, B$ be $G$-$C_0(X)$-algebras.
	Then the representable $KK$-theory $\cR KK(X; A, B)$ is defined to be the group of equivalence classes of Kasparov $(A, B)$-cycles $(\cE, \phi,T)$, defined previously, that satisfy the additional condition
	\[
	(fa)eb=ae(fb), \qquad \forall f\in C_0(X), a\in A, b\in B, e\in\cE,
	\]
where the equivalence relation is identical to that defining $KK(A,B)$. The proof we give makes use of the following technical result of Kasparov.
	\begin{theorem}[{\cite[Theorem 3.4]{Kasparov}}]
		\label{thm:Kas}
		Let $X$ be a $\sigma$-compact space on which groups $G$ and $\Gamma$ act and assume that these actions commute. Suppose also that the $\Gamma$-action is proper and free. Then for any $G$-$C_0(X)$-algebras $A, B$, the descent map gives rise to an isomorphism
		\[
		\cR KK^{G\times\Gamma}(X; A, B)\simeq \cR KK^G(X/\Gamma; A^{\Gamma}, B^{\Gamma}).
		\]
		Here, $A^{\Gamma}$ is a ``fixed-point subalgebra" of $A$ under $\Gamma$, defined in \cite{Kasparov}.
	\end{theorem}
	We are now in a position to carry out the proof of the main result of this subsection.
	\begin{proposition}
		\label{prop:ana.ind.iso}
		Let $G$ be an almost-connected Lie group acting properly on $X$ and $K$ a maximal compact subgroup. For $Y$ a global $K$-slice of $X$, we have
		\[
		K^K_{\bullet}(Y)\simeq K^G_{\bullet+d}(X).
		\]
	\end{proposition}
	\begin{proof}
		Note that for the $\mathbb{C}$-algebra $C(Y)$, the definitions of $\cR KK$ and $KK$ coincide:
		\begin{align*}
		K^K_\bullet(Y) &\simeq \cR KK^K_\bullet(pt;C(Y),\mathbb{C}).
		\end{align*}
		Now a manifold is a $\sigma$-compact space, and the action of $G$ on itself is proper and free. So from Theorem~\ref{thm:Kas} we obtain (noting that $C_0(G)^G$ is precisely $\mathbb{C}$ in the sense of \cite{Kasparov})
		\[
		\cR KK^K_{\bullet}(pt; C(Y), \C)\simeq \cR KK_{\bullet}^{G\times K}(G; C_0(G\times Y), C_0(G)).
		\]
		Applying Theorem~\ref{thm:Kas} again to the right-hand side gives us
		\[
		\cR KK_{\bullet}^{G\times K}(G; C_0(G\times Y), C_0(G))\simeq \cR KK_{\bullet}^{G}(G/K; C_0(G\times_K Y), C_0(G/K)).
		\]
		Therefore, 
		\begin{equation}
		\label{eq:1}
		K^K_{\bullet}(Y)\simeq \cR KK_{\bullet}^{G}(G/K; C_0(G\times_K Y), C_0(G/K))
		=KK_{\bullet}^{G}(C_0(G\times_K Y), C_0(G/K)),
		\end{equation}
		noting that $G/K$ is contractible to a point. Because $G$ is almost-connected, it is a {\em special manifold} in the sense of ~\cite{Kasparov}, so that there exist a Dirac element and a Bott element:  
		\[
		[\partial_{G/K}]\in KK_d^G(C_0(G/K), \C), \qquad [\beta]\in KK_d^G(\C, C_0(G/K)),
		\]
		such that (denoting by $\otimes$ the Kasparov product between classes,  cf. ~\cite{Kasparov})
		\begin{align*}
		[\partial_{G/K}]\otimes_\C[\beta]=1 &\in KK^G(C_0(G/K), C_0(G/K)),\\
		[\beta]\otimes_{C_0(G/K)}[\partial_{G/K}]=1 &\in KK^G(\C, \C).
		\end{align*}
		This leads to the isomorphism 
		\[
		KK^{G}_{\bullet}(C_0(G\times_K Y), C_0(G/K))\simeq KK^{G}_{\bullet+d}(C_0(G\times_K Y),\mathbb{C}).
		\] 
		Together with~(\ref{eq:1}) we obtain $K^K_{\bullet}(Y)\simeq K^G_{{\bullet}+d}(X).$
		\color{black}
	\end{proof}
	
	Now we give an explicit map for this isomorphism. From the proof of the above Proposition, we find that the image under $j$ of a $K$-equivariant Kasparov cycle $[(\cH, \phi, F)] \in K^{K}_{\bullet}(Y)$ is obtained by taking the Kasparov product of the ``lifted" cycle 
	\begin{equation}
	\label{eq:lifted.cycle}
	[(C_0(G\times_K \cH), \tilde\phi, \tilde F)] \in KK^{G}_{\bullet}(C_0(G\times_K Y), C_0(G/K))
	\end{equation}
	with the Dirac element $[\partial_{G/K}]\in KK_d^G(C_0(G/K), \C)$.
	Note that $C_0(G\times_K \cH)$ is a $G$-Hilbert $C_0(G/K)$-module whose $C_0(G/K)$-valued inner product is given by the fibrewise inner product on $\cH$, and $\tilde\phi$ is the multiplication action on $C_0(G\times_K\cH)$ and $\tilde F$ is a family of operators indexed by $G/K$ and given on each fibre $\cH$ by $F$.
	
	The proof that $j$ is an isomorphism did not make use of a $G$-equivariant $\Spinc$-structure on $G/K$, but for the remainder of this paper, we shall assume that such a structure on $G/K$ exists.
	Let $\mathfrak g, \mathfrak k$ be the Lie algebras of $G$ and $K$. There is a Lie algebra $\mathfrak p$ such that the splitting $\mathfrak g=\mathfrak k\oplus\mathfrak p$ is invariant 
	under the adjoint action of $K.$ Our assumption of a $G$-equivariant $\Spinc$-structure on $G/K$ means that 
	$\Ad: K\rightarrow SO(\mathfrak p)$ can be lifted to
	\begin{equation} 
	\label{eq tilde Ad}
	\widetilde{\Ad}: K \to \Spin^c(\kp).
	\end{equation}  
	\begin{remark}
		\label{rem:G/Kspinc}
		Replace $G$ by a double cover $\widetilde G$ and consider the diagram
		\[
		\xymatrix{
			\widetilde{K} \ar[r]^-{\widetilde{\Ad}} \ar[d]_{\pi_K}& \Spin^c(\kp) \ar[d]_{\pi}^{}\\
			K \ar[r]^-{\Ad} & \SO(\kp),
		}
		\]
		where
		\[
		\widetilde{K} := \{ (k, a) \in K\times \Spin^c(\kp); \Ad(k) = \pi(a)\},
		\]
		and the maps $\pi_K$ and $\widetilde{\Ad}$ are defined by
		\[
		\begin{split}
		\pi_K(k, a)&:= k; \\
		\widetilde{\Ad}(k, a) &:= a,
		\end{split}
		\]
		for $k \in K$ and $a \in \Spin^c(\kp)$.
		Then $\widetilde G/\widetilde K$ has a $G$-equivariant $\Spinc$-structure. Indeed, for all $k \in K$,
		\[
		\pi_K^{-1}(k) \cong \pi^{-1}(\Ad(k)) \cong U(1),
		\]
		so $\pi_K$ is the projection of a $U(1)$-central extension. Since $G/K$ is contractible, $\widetilde{K}$ is the maximal compact subgroup of a $U(1)$-central extension of $G$.
	\end{remark}
	
	Denote by $S$ the $K$-vector space underlying $\mathfrak{p}$ in the $\Spinc$-representation \eqref{eq tilde Ad} of $K$.
	Fix the normalising function
	\[
	b(x)=\frac{x}{\sqrt{x^2+1}}.
	\]
	The Dirac element can be written as 
	\[
	[\partial_{G/K}]:=\left[\left((L^2(G)\otimes S)^K, m, b(\partial_{G/K})\right)\right],
	\]
	where $m$ is scalar multiplication of $C_0(G/K)$ on the Hilbert space $(L^2(G)\otimes S)^K$ and $\partial_{G/K}$ is the $\Spinc$-Dirac operator on $G/K$.
	Then the image of $[(\cH, \phi, F)]$ under induction is 
	\[
	j[(\cH, \phi, F)]=[(\cE, \tilde\phi, \tilde F\sharp b(\partial_{G/K}))],
	\] 
	where $\cE$ is the Hilbert space
	\[
	\cE:=C_0(G\times_K \cH)\otimes_{C_0(G/K)}(L^2(G)\otimes S)^K\simeq (L^2(G)\otimes\cH\otimes S)^K,
	\]
	$\tilde\phi$ is the representation 
	\[
	\tilde\phi: C_0(G\times_K Y)\rightarrow \cL(\cE)
	\] 
	given by the obvious pointwise multiplication determined by $\phi: C(Y)\rightarrow\cL(\cH)$,
	and 
	\[
	\tilde F\sharp b(\partial_{G/K})
	\]
	means the Kasparov product (of operators) of $\tilde F$ and $b(\partial_{G/K})=\partial_{G/K}(1+\partial_{G/K}^2)^{-\frac12}$.
	When $F$ is also a Dirac-type operator, the Kasparov product can be explicitly written down, as follows.
	
	\begin{example}
		Let $Y$ be a $K$-manifold with a $K$-equivariant $\Spinc$-structure and $\partial_Y$ be the associated $\Spinc$-Dirac operator.
		Then 
		\begin{align*}
		\cE& =C_0(G\times_K L^2(Y, S_Y))\otimes_{C_0(G/K)}(L^2(G)\otimes S)^K\simeq (L^2(G)\otimes L^2(Y, S_Y)\otimes S)^K\\
		& \simeq L^2(G\times_K Y, G\times_K(S_Y\times S))\simeq L^2(X, S_X).
		\end{align*}
		and $b(\partial_X)$ represents the Kasparov product of $b(\tilde\partial_Y\otimes 1)$ and $b(1\otimes\partial_{G/K})$, namely
		\[
		\left(\frac{\tilde\partial_Y}{\sqrt{\tilde\partial_Y^2+1}}\otimes 1\right)\sharp\left(1\otimes \frac{\partial_{G/K}}{\sqrt{1+\partial_{G/K}^2}}\right)=\frac{\partial_X}{\sqrt{1+\partial_X^2}}.
		\] 
		Therefore
		\begin{equation}
		\label{eq:Khomology.XY}
		j[\partial_Y]=[\partial_X].
		\end{equation}
		This relation also holds for twisted Dirac operators (cf. the proof of Lemma~\ref{lem:comm.diagram}).
	\end{example}


	\subsection{Commutativity of the Diagram}
	\label{sec:Commutativity of the diagram}
	
	\begin{lemma}
		\label{lem:comm.diagram}
		Diagram~(\ref{eq:comm.K.homo}) commutes.
	\end{lemma}
	
	\begin{proof}
		Let $x$ be an element of $K^K_{\bullet}(Y)$. 
		From~\cite{BOSW}, there is a geometric cycle $(N, E, f)$ representing an element of $K^{\textnormal{{geo,}} K}_{\bullet}(Y)$ such that 
		\[
		x=f_*([E]\cap[D_N]).
		\]
		By definition, the map $i$ sends the class of geometric cycles $[(N, E, f)]$ to the class of geometric cycles 
		\[
		[(M, \tilde E, \tilde f)]\in K^{\textnormal{{geo,}} G}_{\bullet+d}(X),
		\]
		where $M=G\times_K N$, $\tilde E=G\times_K E$ and $\tilde f: M\rightarrow X$ is the lift of $f: N\rightarrow Y$ as in Definition \ref{def:geometricinduction}.
		Thus, elements in~(\ref{eq:comm.K.homo}) are related as follows:
		\begin{equation*}
		\label{eq:comm.K.homo.ele}
		\begin{CD}
		[(N, E, f)] @>BD>> f_*([E]\cap[D_N])\\
		@ViVV @VjVV\\
		[(M, \tilde E, \tilde f)] @>BD>> \tilde f_*([\tilde E]\cap[D_N]),
		\end{CD}
		\end{equation*}
		and commutativity of~(\ref{eq:comm.K.homo}) means precisely
		\[
		j(f_*([E]\cap[D_N]))=\tilde f_*([\tilde E]\cap[D_M]).
		\]
		Writing both sides in terms of cycles, we need to show that the Kasparov $(C_0(X), \C)$-cycles
		\[
		((L^2(G)\otimes L^2(S_N\otimes E)\otimes S)^K, \widetilde{m\circ f'}, b(\widetilde{D_{N, E}})\sharp b(\partial_{G/K}))
		\]
		and 
		\[
		(L^2(S_M\otimes\tilde E), \tilde m\circ \tilde f', b(D_{M, \tilde E}))
		\]
		give rise to the same element in $K^G_0(X).$
		Here $S_N$ stands for the spinor bundle over $N$ and $m$ stands for the pointwise multiplication of the algebra $C_0(N)$ of continuous functions on the Hilbert space. At this point, recall the well-known fact that the $K$-homology class of a Dirac operator is independent of the choice of metric. Given a $K$-invariant metric on $N$, let us pick the natural ``product-type" metric on $G\times_K N$. Then we have
		\[
		(L^2(G)\otimes L^2(S_N\otimes E)\otimes S)^K\simeq L^2(S_M\otimes\tilde E),
		\]
		since $S_M=G\times_K (S_M|_N)$ and the restriction $S_M$ to $N$ splits into a tensor product of $S_N$ and $S$.  
		Obviously, $\tilde m\circ \tilde f'$ and $\widetilde{m\circ f'}$ both give the same representation of $C_0(X)$ on $L^2(S_M\otimes\tilde E)$ by scalar multiplication, compatible with $\tilde f': C_0(X)\rightarrow C_0(M).$
		Finally, from $D_{N, E}$ to $\widetilde{D_{N, E}}$ we replace the $K$-equivariant $\Spinc$-connection $\nabla^{E}$ on $N$ by the $G$-equivariant connection $\nabla^{\tilde E}=G\times_K\nabla^{E}$ on $M$. 
		If $\pi: M=G\times_K N\rightarrow G/K$ is the projection, the $\Spinc$-connection $\nabla^{G/K}$ on $G/K$ is pulled back to a $G$-equivariant $\Spinc$-connection $\pi^*\nabla^{G/K}$ on $M.$
		Let $\{e_i\}$ be a local orthonormal frame in the direction of fibres of $G\times_KN\rightarrow G/K$ and $\{f_j\}$ a local orthonormal frame of the base $G/K.$
		Then $b(\widetilde{D_{N, E}})\sharp b(\partial_{G/K})$ looks locally like $b(x)$, where
		\[
		x=\sum_{i=1}^{\dim N}c(e_i)\nabla^{\tilde E}_{e_i}+\sum_{j=1}^{d}c(f_j)\pi^*\nabla^{G/K}_{f_j}.
		\]
		With respect to the fibration $N\rightarrow M\rightarrow G/K$, a $G$-invariant metric $g_M$ on $M$ determines a metric $g_{G/K}$ on the base $G/K$ and a family of metrics $g_{M/(G/K)}$ on the family of manifolds $\pi: M\rightarrow G/K$ parametrised by $G/K$. Taking the adiabatic limit in the parameter $s$,
		$$g_{M, s}:=g_{M/(G/K)}+s^{-2}\pi^*g_{G/K}$$ approaches a product as $s\to 0$, and these metrics are $G$-invariant (see \cite{Bismut}). 
		Since the $K$-homology class represented by a Dirac operator is independent of the choice of metrics, the operator $b(\widetilde{D_{N, E}})\sharp b(\partial_{G/K})$ represents the same $K$-homology class as $b(D_{M, \tilde E}).$
		The lemma is then proved.
	\end{proof}
	
	\remark
	\label{sec:Consequence}
	
	Let us summarise of the consequences of Theorem~\ref{thm:eq.K-homology}. Let $G$ be an almost-connected Lie group acting on a manifold $X$ properly and cocompactly. 
	For every element $x$ in the analytic $K$-homology $K_0^G(X)$, there exists a unique class $[(M, E, f)]\in K_0^{G, geo}(X)$ such that
	\[
	x=f_*([E]\cap[D_M])=f_*([D_{M, E}]),
	\] 
	where $f: M\rightarrow X$. 
	Moreover, letting $K$ be a maximal compact subgroup, there exists a $K$-submanifold $N$ such that: 
	\begin{itemize} 
		\item $M$ is diffeomorphic to $G\times_K N$,
		\item $N$ admits a $K$-equivariant $\Spinc$-structure compatible with the $G$-equivariant $\Spinc$-structure on $M$ and
		\item $(N, E|_N, f|_N)$ is a geometric cycle for $K_{d}^{K, geo}(Y)$,
	\end{itemize}
	and there exists a unique element $(f|_N)_*([D_{N, E|_{N}}])\in K_d^{K}(Y)$ satisfying
	\begin{equation}
	\label{eq:K-homo.ind}
	K_d^K(Y)\rightarrow K_0^G(X), \qquad (f|_N)_*([D_{N, E|_{N}}])\mapsto f_*([D_{M, E}])=x.
	\end{equation}
	Here $Y=f(N)\subset X$ is a $K$-slice of $X$ and $d$ is the (mod 2) dimension of $G/K$.
	
	\begin{remark}
		\label{rmk:peter's.result}
		In Theorem 4.6 in \cite{HochsDS} and Theorem 4.5 in \cite{HochsPS}, a map
		\[
		\KInd_K^G\colon K_{\bullet}^K(Y) \to K_{\bullet+d}^G(X)
		\]
		is constructed by a different method.
		In Section 6 of \cite{HochsDS}, it is shown that the $K$-homology class of a $\Spinc$-Dirac operator on $Y$, associated to a connection $\nabla^Y$ on the determinant line bundle of a $\Spinc$-structure, is mapped to the class  of a $\Spinc$-Dirac operator on $X$ associated to a connection $\nabla^X$ induced by $\nabla^Y$ on the determinant line bundle of the induced $\Spinc$-structure, by the map $\KInd_K^G$:
		\[
		\KInd_K^G[\partial_Y] = [\partial_X].
		\] 
	\end{remark}
	
	Since we have shown that the two $K$-homology theories are isomorphic and that the diagram (2.5) commutes, the two induction isomorphisms $i$ and $j$ can be thought of as being the same map. In view of Remark~\ref{rmk:peter's.result}, we shall denote both $i$ and $j$ by $\KInd_K^G$, which we refer to as \emph{induction on $K$-homology}. (\ref{eq:K-homo.ind}) can now be formulated in the following way:
	
	\begin{proposition}
		Any $x\in K_0^G(X)$ can be represented by a $G$-equivariant $\Spinc$-Dirac operator on $M$ twisted by a $G$-vector bundle $E$, where $f: M\rightarrow X$ is a continuous $G$-equivariant map such that 
		\[
		x=f_*([D_{M, E}])=\KInd_K^G((f|_N)_*([D_{N, E|_N}])).
		\] 
	\end{proposition}
	
	\vspace{0.5cm}
	\section{Poincar\'e Duality}
	\label{sec:PD}
	\subsection{Almost-connected Groups}
	We begin by remarking that Phillips' \cite{Phillips2}
	generalisation of equivariant $K$-theory, denoted $\bar K^\bullet_G(X)$, which is defined using finite-dimensional 
	equivariant vector bundles over $X$, is not necessarily the
	same as the definition via $C^*$-algebras, namely $K^\bullet_G(X):= K_\bullet(C_0(X) \rtimes G).$ However, in the case of almost-connected Lie groups, these groups are isomorphic,
	as will be argued presently.
	
	\begin{lemma}\label{lem:phil}
		Let $G$ be an almost-connected Lie group
		acting properly and cocompactly on a smooth manifold $X$. Then 
		\begin{equation}
		\label{eq:K-theory.iso}
		\bar K^\bullet_G(X) \simeq  K^\bullet_G(X) = K_\bullet(C_0(X) \rtimes G).
		\end{equation}
	\end{lemma}
	
	\begin{proof}
		Without loss of generality, assume $\bullet = 0$. Under the hypotheses of the lemma, $C_0(X) \rtimes G$ and $C_0(Y) \rtimes K$ are strongly
		Morita-equivalent, by Rieffel-Green \cite{Rieffel}, where $Y$ is a global slice given by Abels' theorem \cite{Abels}.
		By Green-Julg \cite{Julg}, one has 
		$$
		K_0(C_0(Y) \rtimes K) \simeq K_K^0(Y).
		$$
		By definition, 
		$$
		K^0_G(X) = K_0(C_0(X) \rtimes G),
		$$
		therefore
		$$
		K^0_G(X) \simeq K_K^0(Y).
		$$
		But Phillips' \cite{Phillips2} proves that $\bar K^0_G(X) \simeq K_K^0(Y)$, so we conclude.
	\end{proof}
	
	\begin{remark}
		We do not need to specify whether or not the crossed product $C_0(X)\rtimes G$ is reduced, as the action of $G$ on $X$ is proper. See also Remark~\ref{rmk:prop.crossed.prod}.
	\end{remark}
	
	\begin{remark}
		Concretely, the isomorphism \eqref{eq:K-theory.iso} can be expressed as the composition
		\[
		\bar  K^0_G(X)\rightarrow KK^G(C_0(X), C_0(X))\rightarrow KK(C_0(X)\rtimes G, C_0(X)\rtimes G)\rightarrow K_0(C_0(X)\rtimes G),
		\]
		where the first map takes a $G$-equivariant vector bundle to its continuous sections, 
		the second is the descent map~(\ref{eq:descent}) in $KK$-theory and the last is left multiplication via $KK$-product by the canonical projection \eqref{eq:projection} in $C_0(X)\rtimes G$.
		Given a $G$-equivariant vector bundle $V$ over $X$, the images in the above sequence of maps are 
		\[
		[V]\mapsto[(\Gamma(V), 0)]\mapsto[(\Gamma(V)\rtimes G, 0)]\mapsto[p\cdot(\Gamma(V)\rtimes G)].
		\]
		The induction on $\bar K^0$ is simpler than the induction on $K^0$. 
		In fact, if $E$ is a $K$-equivariant vector bundle over $Y$ and $p^K$, $p^G$ are the canonical projections in $C(Y)\rtimes K$, $C_0(X)\rtimes G$ respectively, then in the diagram
		\begin{equation}
		\label{eq:ind.barK.K}
		\begin{CD}
		\bar K_{K}^0(Y) @>>> K_K^0(Y)\\
		@VVV @VVV\\
		\bar K_{G}^0(X) @>>> K_G^0(X)
		\end{CD}
		\end{equation}
		we have 
		\begin{equation}
		\label{eq:ind.barK.K.ele}
		\begin{CD}
		[V] @>>> [p^K\cdot(\Gamma(V)\rtimes K)]\\
		@VVV @VVV\\
		[G\times_K V] @>>> [p^G\cdot(\Gamma(G\times_K V)\rtimes G)].
		\end{CD}
		\end{equation}
		In particular, if $V$ is the rank $1$ trivial vector bundle over $Y$, then $G\times_K V$ is also a rank $1$ vector bundle over $X$ and they correspond to the canonical projections $[p^K]\in K^0_K(Y)$ and $[p^G]\in K^0_G(X)$ respectively.
		The induction map $K_K^0(Y)\rightarrow K_G^0(X)$ can also be understood using the isomorphism  
		\begin{equation}
		\label{eq:GXKX}
		K_0(C_0(X)\rtimes G)\simeq K_d(C_0(X)\rtimes K)
		\end{equation}
		proved in~\cite{Raeburn} and the fact that $X$ is $K$-homeomorphic to $Y\times\R^d$, where $K$ acts on $\R^d$ via the diffeomorphism $\R^d\cong G/K$ (see Remark~5.19 in~\cite{Raeburn}).
		Finally, the induction maps \eqref{eq:ind.barK.K}-\eqref{eq:ind.barK.K.ele} imply that both $K^0_G(X)$ and $\bar K^0_G(X)$ are $R(K)$-modules and that~(\ref{eq:K-theory.iso}) is an isomorphism of $R(K)$-modules.
	\end{remark}
	
	We now prove an equivariant Poincar\'e duality under the same hypotheses as above.

	\begin{theorem}[Poincar\'e duality: cocompact case] \label{thm:PD1} Let $G$ be an almost-connected Lie group
		acting properly and cocompactly on a smooth manifold $X$. 
		Then there are isomorphisms
		\begin{align}\label{PD}
		{\mathcal{P}\mathcal{D}_X}_*: K^G_{\bullet}(C_\tau(X)) &\simeq K^G_{\bullet}(X);\\
		\mathcal{P}\mathcal{D}_X^*: K_G^{\bullet}(C_\tau(X)) &\simeq K_G^{\bullet}(X) \label{PD'},
		\end{align}
		where $C_\tau(X)$ is the algebra of continuous sections, tending to 0 at $\infty$, of the complex Clifford bundle associated with 
		the tangent bundle $TX$ of $X$. 
	\end{theorem}
	
	\begin{proof}
		We use  Abels' global slice theorem to see that $X$ is diffeomorphic to  $G\times_K Y$, where 
		$K$ is a maximal compact subgroup of $G$ and $Y$ is a smooth compact manifold. 
		Using Morita equivalence of $C_{\tau}(X)$ and $C_0(TX)$~{\cite[Theorem~2.7]{Kasparov2016}}, the decomposition 
		\begin{equation}
		\label{eq:decom.tangent}
		TX=G\times_K[TY\oplus\mathfrak p],
		\end{equation}
		where $\mathfrak p\oplus\mathfrak k=\mathfrak g$, and Phillips' result \cite{Phillips1,Phillips2} proving that induction from $K$ to $G$ in $K$-theory is an isomorphism,
		we obtain \begin{equation}
		\label{eq:K.theory.ind}
		K^{K}_{\bullet}(C_\tau(Y)) \simeq K^{G}_{\bullet+d}(C_\tau(X)).
		\end{equation}
		Then, \eqref{eq:K.theory.ind} using equivariant Poincar\'e duality in the compact case (see \cite{Kasparov}), 
		$$
		{\mathcal P\mathcal D}: K^K_{\bullet}(C_\tau(Y))\simeq K^K_{\bullet}(Y),
		$$
		together with analytic induction from $K$ to $G$, which is an isomorphism
		$$
		K^{K}_{\bullet}(Y) \simeq K^{G}_{\bullet+d}(X),
		$$
		we deduce the isomorphism \eqref{PD}. The second isomorphism is proved analogously.
	\end{proof}
	
	We note that the above proof only used induction on the {\it analytic} version of $K$-homology. However, the equivalence between the two $K$-homologies can be used to give the following geometric interpretation of Poincar\'e duality.

	\begin{corollary}
		\label{cor:spinc}
		Let $G$ be an almost-connected Lie group acting properly and cocompactly on a manifold $X$. If $X$ and $G/K$ admit  $G$-equivariant $\textnormal{Spin}^c$-structures, we have
		\begin{equation}
		\label{eq:PD.geo.K}
		\mathcal{P}\mathcal{D}_X: K^0_G(X)\rightarrow K_a^{G, geo}(X), \quad [E]\mapsto [X]\cap [E]:=[X,E\otimes\C,\id_X],
		\end{equation}
		where $E$ is a finite-rank $G$-equivariant vector bundle over $X$ and $a=\dim X$ (mod $2$).
	\end{corollary}
	\begin{proof}
		Since $X$ is a proper $G$-cocompact $\Spinc$-manifold, $C_{\tau}(X)$ is Morita-equivalent to $C_0(X)$. Thus we may write
		\begin{equation}
		\mathcal{P}\mathcal{D}_X: K^0_G(X)\rightarrow K_a^{G}(X).
		\end{equation}
		Then the claim follows by the equivalence between geometric and analytic $K$-homologies given in Theorem~\ref{thm:eq.K-homology}.
	\end{proof}
	
	\begin{remark}
		When $X$ is a proper and cocompact $G$-manifold, not necessarily $\Spinc$, the first isomorphism ~(\ref{PD}) is equivalent to 
		\begin{equation}
		\label{eq:op.sym}
		K^G_{\bullet}(X)\simeq K^G_{\bullet}(C_0(TX))\simeq K_{\bullet}(C_0(TX)\rtimes G).
		\end{equation}
		This isomorphism is closely related to a generalisation of the Atiyah-Singer index formula, since it is an operator-to-symbol map.
		In fact, recall that for a $K$-slice $Y$ of $X$, $K$-invariant pseudo-differential operators represent classes in $K^K_{\bullet}(Y)$, while their symbols give rise to classes in $K^K_{\bullet}(C_0(TY))$ (cf. {\cite[Section~5]{AS}}). 
		Then, for a class of $G$-invariant pseudo-differential operators, which we shall denote by $[D_X]\in K^G_{\bullet}(X)$, with symbol class $[\sigma(D_X)]\in KK^G_{\bullet}(C_0(X), C_0(TX))$ (cf. \cite{Kasparov2016}),  
		the map 
		\[
		[D_X]\mapsto [p]\otimes_{C_0(X)\rtimes G}j^G[\sigma(D_X)]
		\] 
		realises the isomorphism~\eqref{eq:op.sym}.
		Note that using Theorem~\ref{thm:eq.K-homology} and the commutative diagram
		\[
		\begin{CD}
		K_0^G(M) @>>> K_0^G(C_0(T^*M))\\
		@Vf_*VV @Vf_*VV\\
		K_0^G(X) @>>> K_0^G(C_0(T^*X)),
		\end{CD}
		\]
		we see that every element of $K^G_{\bullet}(X)$ is represented by a $G$-invariant $\Spinc$-Dirac operator.
	\end{remark}
	
	\begin{remark}
		When $X$ is a proper and cocompact $G$-manifold, not necessarily $\Spinc$, the second isomorphism~(\ref{PD'}) maps a $G$-equivariant vector bundle $[E]\in K_G^0(X)$ to
		\[
		[d_{X, E}]=[d_X]\cap[E]\in K^0_G(C_{\tau}(X)),
		\] 
		where $[d_X]\in K^0_G(C_{\tau}(X))$ is the Dirac element defined using the de Rham operator on $X$ in~\cite{Kasparov}.
		This can be easily verified using induction: 
		\[
		\begin{CD}
		K_0^K(Y) @>>> K^0_K(C_{\tau}(Y))\\
		@VVV @VVV\\
		K_0^G(X) @>>> K^0_G(C_{\tau}(X)),
		\end{CD}
		\]
		given by 
		\[
		\begin{CD}
		[E|_Y] @>>> [d_Y]\cap[E|_Y]\\
		@VVV @VVV\\
		[E] @>>> [d_X]\cap[E],
		\end{CD}
		\]
		with the help of the fact that $[d_X]$ is mapped to $[d_Y]$ under $K$-homology induction adapted to Clifford algebras.
	\end{remark}
	
	\begin{remark}
		From~{\cite[Theorem~4.1]{Phillips2}}, for any proper $G$-space $X$ (not necessarily cocompact), there is an isomorphism  
		\[
		K^{\bullet}_G(X)\simeq K_K^{\bullet}(X\times\mathfrak g/\mathfrak k)\simeq K_K^{\bullet+d}(X).
		\]
		This is the same map as~\eqref{eq:GXKX}.
		Replacing $X$ by $TX$ and using $KK$-equivalence of $C_0(TX)$ and $C_{\tau}(X)$ and applying Theorem~4.11 of {\cite{Kasparov}} and Theorem~\ref{thm:PD1}, we obtain the dual version of Phillips' isomorphism for a proper \emph{cocompact} $G$-space $X$:
		\[
		K_{\bullet}^G(X)\simeq RK^K_{\bullet+d}(X) \simeq RK^K_{\bullet}(X\times\mathfrak g/\mathfrak k).
		\]
	\end{remark}
	
	For compact manifolds $\bar Y$ with boundary $\partial \bar Y \ne \emptyset$, Poincar\'e duality is proved in \cite{Kasparov}:

	\begin{theorem} [Poincar\'e duality for $K$-compact manifolds with boundary]\label{thm:PD2}
		Assume that $\bar Y$ is a smooth compact manifold with boundary $\partial \bar Y$ and that a compact group $K$ acts on $\bar Y$ smoothly. Set $Y = \bar Y\setminus \partial \bar Y$. Then there are isomorphisms
		\begin{align*}
		{\mathcal{P}\mathcal{D}_Y}_*:  K^K_{\bullet}(C_\tau(Y)) & \simeq K^K_{\bullet}(\bar Y);\\
		\mathcal{P}\mathcal{D}_Y^*:  K_K^{\bullet}(C_\tau(Y)) & \simeq K_K^{\bullet}(\bar Y).
		\end{align*}
	\end{theorem}
	

	The proof of the following theorem is similar to that of Theorem \ref{thm:PD1} but using instead Theorem \ref{thm:PD2}  and will be omitted.
	Observe that by Abels' global slice theorem, $\bar X = G\times_K \bar Y,$  $\partial \bar X = G\times_K \partial\bar Y,$
	and  $X = G\times_K Y.$
	
	\begin{theorem} [Poincar\'e duality for $G$-cocompact manifolds with boundary]\label{thm:PD3}
		Assume that $\bar X$ is a smooth $G$-cocompact manifold with boundary $\partial \bar X$ and that an almost-connected Lie group $G$ acts on $\bar X$ smoothly. Set $X = \bar X\setminus \partial \bar X$. Then there are isomorphisms
		\begin{align*}
		{\mathcal{P}\mathcal{D}_X}_*:  K^G_{\bullet}(C_\tau(X)) & \simeq K^G_{\bullet}(\bar X);\\
		\mathcal{P}\mathcal{D}_X^*:  K_G^{\bullet}(C_\tau(X)) & \simeq K_G^{\bullet}(\bar X).
		\end{align*}
	\end{theorem}
	
	We now generalise Poincar\'e duality to the case 
	when $X$ is not necessarily $G$-cocompact.
	
	Recall that the representable equivariant $K$-theory of a $G$-space $X$, as defined by Fredholm complexes in \cite{Segal}, denoted by $RK^\bullet_G(X)$, is equal to $K^G_\bullet(C_0(X))$ when $X$ is $G$-cocompact, and is defined as the direct limit
	$$
	RK^\bullet_G(X) :=\lim_{Z\subset X} K^G_\bullet(C_0(Z))
	$$
	over the inductive system of all cocompact $G$-subsets $Z \subset X$.

	Theorem~\ref{thm:PD3} allows us to prove Poincar\'e duality for non-cocompact manifolds, the main result of this section:
	
	\begin{theorem} [Poincar\'e duality for non-cocompact manifolds]
		\label{thm:PD.NC}
		Assume that $X$ is a complete Riemannian manifold on which an almost-connected Lie group $G$ acts isometrically.
		Then one has isomorphisms
		\begin{align}
		{\mathcal{P}\mathcal{D}_X}_*:  K^G_{\bullet}(C_\tau(X)) & \simeq RK^G_{\bullet}(X);\\
		\mathcal{P}\mathcal{D}_X^*:  K_G^{\bullet}(C_\tau(X)) & \simeq RK_G^{\bullet}(X),
		\label{eq:2}
		\end{align}
		where the right-hand side denotes the representable versions of equivariant $K$-homology and $K$-theory.
	\end{theorem}
	
	\begin{proof}
		We sketch the proof here. 
		Consider an exhaustive increasing
		sequence of cocompact $G$-manifolds with boundary $\bar X_j$
		such that $X=\bigcup \bar X_j$.
		Theorem \ref{thm:PD3} gives us a coherent system of isomorphisms:
		\begin{align*}
		{\mathcal{P}\mathcal{D}_j}_*:  K^G_{\bullet}(C_\tau(X_j)) & \simeq K^G_{\bullet}(\bar X_j);\\
		\mathcal{P}\mathcal{D}_j^*:  K_G^{\bullet}(C_\tau(X_j)) & \simeq K_G^{\bullet}(\bar X_j).
		\end{align*}
		The isomorphism ${\mathcal{P}\mathcal{D}_X}_*$ is obtained as the direct limit isomorphism of the first
		coherent system of isomorphisms above, and 
		$\mathcal{P}\mathcal{D}_X^*$ is Milnor's $\lim^1$ inverse limit of the second coherent system of isomorphisms above.
	\end{proof}
	
	Our result is related to the Poincar\'e duality in~{\cite[Section~4]{Kasparov}} as follows.
	
	\begin{remark}
		Theorem~\ref{thm:PD.NC} is a generalisation of Corollary~4.11 in~\cite{Kasparov}. 
		Kasparov's first Poincar\'e duality ({\cite[Theorem~4.9]{Kasparov}}) states that for any locally compact group $G$,
		\[
		R KK^G(Y\times X; A, B)\simeq R KK^G(Y; A\hat\otimes C_{\tau}(X), B),
		\]
		where $X$ is a complete Riemmanian manifold with an isometric $G$-action, $Y$ is a $\sigma$-compact $G$-space and $A, B$ are separable $G$-$C^*$-algebras. 
		Here, 
		$$RKK(X; A, B):=\cR KK(X; A(X), B(X))$$
		(cf. {\cite[Section~2.19]{Kasparov}}).
		When $A=B=\C$ and $Y$ is a point, this isomorphism reduces to the second isomorphism of Theorem~\ref{thm:PD.NC}.
	\end{remark}
	
	We end this section by generalising the Atiyah-Segal completion theorem~\cite{ASe} using induction on $K$-homology and Poincar\'e duality.
	
	\begin{remark}
		Let $K$ be a compact Lie group and $A, B$ be $G$-$C^*$-algebras. $KK^K(A, B)$ is an $R(K)$-module. Denote by $KK^K(A, B)^{\wedge}$ the $I(K)$-adic completion of $KK^K(A, B)$ in the sense of~\cite{ASe}. 
		In~{\cite[Theorem~3.19]{AK}}, a generalisation of the Atiyah-Segal completion theorem is proved:
		\[
		KK^K(A, B)^{\wedge} \simeq RKK^K(EK; A, B), 
		\]
		where $KK^K(A, B)$ is a finite $R(K)$-module. 
		Note that choosing finite $R(K)$-modules $A=\C$ and $B=C(Y)$ for a compact $K$-manifold $Y$, we obtain\[
		RKK^K(EK; \C, C(Y))\simeq RK^0(EK\times_K Y),
		\] 
		and the Atiyah-Segal completion theorem~\cite{ASe} is recovered. 
		
		Now assume $A, B$ to be $G$-algebras.
		Note that $EG=G\times_K EK$ and that there is the following induction isomorphism, from Section 3.6 of~\cite{Kasparov}: 
		\[
		RKK^K(EK; A, B)\simeq RKK^G(G\times_K EK; A, B).
		\] 
		\begin{enumerate}
			\item Letting $A=\C$ and $B=C(X)$, we have 
			\[
			KK^K_{\bullet}(\C, C_0(X))\simeq K_G^{\bullet}(X)
			\]
			as $R(K)$-modules. So its $I(K)$-adic completion is isomorphic to $RKK^G(EG; \C, C(X))$, and 
			\[
			K_G^{\bullet}(X)^{\wedge}\simeq RKK^G_{\bullet}(EG; \C, C_0(X))\simeq RK^{\bullet}(EG\times_G X),
			\]
			whence we recover Phillips' result~{\cite[Theorem~5.3]{Phillips2}}.
			\item Letting $A=C(X)$ and $B=\C$, we have from induction on $K$-homology that
			\[
			KK^K_{\bullet}(C_0(X), \C)\simeq K^G_{\bullet}(X)
			\]
			as $R(K)$-modules. So its $I(K)$-adic completion is isomorphic to $RKK^G(EG; C(X), \C)$, and 
			\[
			K^G_{\bullet}(X)^{\wedge}\simeq RKK^G_{\bullet}(EG; C_0(X), \C)\simeq KK_{\bullet}(C_0(EG\times_G X), C_0(BG)),
			\]
			whence we obtain a new Atiyah-Segal completion result for $G$-equivariant $K$-homology.
		\end{enumerate}
		
		Note that if $X$ is $\Spinc$, (1) and (2) are related by Poincar\'e duality (cf.~{\cite[Theorem~4.10]{Kasparov}}). 
	\end{remark}
	
	\subsection{Discrete Groups}
	Let $G$ be a locally compact group and $X$ a proper $G$-manifold where $X/G$ is compact. 
	As mentioned previously, the $K$-theory $K^*_G(X)$ fails to be a generalised cohomology in general. 
	In~\cite{Phillips1}, Phillips constructed a cohomology theory using $G$-vector bundles with Hilbert space fibres and shows in an example that finite-dimensional vector bundles are not enough. The example is a semidirect product of $\Z^2$ by the $4$-torus, which is neither discrete nor linear. 
	However, it has been verified that finite-dimensionality is enough if $G$ is either discrete (Theorem 3.2 in \cite{LO}) or linear (\cite{Phillips2}), and moreover that 
	$K^*_G(X)$ is a generalised cohomology theory in these cases.
	
	Assume $G$ to be a discrete or linear group and $X$ a proper $G$-manifold with compact quotient. Every element of $K$-theory $K^*_G(X)$ is represented by a compact $G$-equivariant vector bundle $E$ over $X$~\cite{LO, Phillips2} for $*=0$. 
	(For $*=1$ a $K$-theory element is given by $G$-vector bundle over the suspension of $X$.)
	Let $D$ be the de Rham operator on $X$ representing the Dirac element $[D]$ in the equivariant $K$-homology $K^*_G(C_{\tau}(X))$ in the sense of~\cite{Kasparov} Definition-lemma 4.2. 
	Twisting $E$ with the de Rham operator $D$ gives rise to a twisted Dirac element and hence the homomorphism below: 
	\begin{equation}
	\label{eq:PD-map}
	\mathcal{P}\mathcal{D}: K^*_G(X)\rightarrow K^*_G(C_{\tau}(X)) \qquad [E]\mapsto[D_E].
	\end{equation}
	
	\begin{theorem}
		\label{thm:PD.discrete}
		Let $G$ be a discrete group acting properly on a $G$ manifold $X$ with compact quotient. The Poincar\'e homomorphism~(\ref{eq:PD-map}) is an isomorphism, defining the Poincar\'e duality $\mathcal P\mathcal D: K^{\bullet}_G(X)\simeq K^{\bullet}_G(C_{\tau}(X)).$
	\end{theorem}
	
	Every proper compact $G$-manifold $X$ is covered by $G$-slices of the form $G\times_H U$ where $H$ is a compact subgroup of $G$ and $U$ is an $H$-space:
	\[
	X=\cup_{i=1}^N G\times_{H_i}U_i.
	\]
	We may assume that all $U_i$s and their nonempty ``intersections" are homeomorphic to $\R^m$ for some $m$. 
	Here, $U_{ij}$ is an nonempty intersection of $U_i$ and $U_j$ means that 
	\[
	[G\times_{H_i}U_i]\cap [G\times_{H_j}U_j]=G\times_{H_{ij}}U_{ij}
	\] 
	for some compact subgroup $H_{ij}$ of $G.$
	
	A cover $\{G\times_{H_i}U_i\}_{i=1}^N$ satisfying the above conditions is called a {\em good cover} of $X.$
	
	As one would expect, we will apply Mayer-Vietoris sequence and the Five Lemma to prove Poincar\'e duality. This means that we need to extend the Poincar\'e duality map to  non-$G$-cocompact manifolds. 
	See~\cite{Bott-Tu} for the proof Poincar\'e duality for ordinary (co)-homology and \cite{Higson-Roe} 11.8.11 for the $K$-theory analogue for compact manifolds in the nonequivariant case. 
	
	For every open proper $G$-manifold $U$, denote by $RK^{\bullet}_{G}(C_{\tau}(U))$ the equivariant $K$-homology with compact support defined by the direct limit over all $G$-compact submanifolds $L$ in $U$ with respect to the restriction map $C_{\tau}(U)\rightarrow C_{\tau}(L):$
	\[
	RK^{\bullet}_{G}(C_{\tau}(U))=\lim_{L\subset M} K^{\bullet}_{G}(C_{\tau}(L)).
	\]
	The equivariant $K$-theory of $U$, by definition, is represented by a $G$-vector bundle $E$ that is trivial outside a $G$-compact subset $C\subset U$. 
	Thus, the Poincar\'e duality map~(\ref{eq:PD-map}) is extended to a non $G$-compact manifold $U\subset M$ by 
	\begin{equation}
	\label{eq:PD.map.open}
	\mathcal P \mathcal D: K^{\bullet}_G(U)\rightarrow RK^{\bullet}_{G}(C_{\tau}(U)) \qquad [E]\mapsto \{[D_E|_{C\cup L}]\}_{L\subset U, L/G \text{  compact}},
	\end{equation}
	where the inductive limit is taken over $G$-compact submanifolds $C\cup L\subset U.$
	
	\begin{proof}[Proof of Theorem~\ref{thm:PD.discrete}]
		Because $G$ is discrete, it follows from~\cite{LO} that the equivariant $K$-theory $K^*_G(X)$ is a generalised cohomology theory. 
		The (analytic) equivariant homology $RK^{\bullet}_G(C_{\tau}(X))$
		is known to be a generalised homology theory from the work of Kasparov.
		Therefore for two open $G$-submanifolds $U$ and $V$ of $X$, we have the Mayer-Vietoris sequences for both $K$-theory and $K$-homology, with the corresponding terms related (vertically in the following diagram) by Poincar\'e duality maps~(\ref{eq:PD.map.open}):
		\[
		\begin{CD}
		\to K^j_G(U\cap V) @>>> K^j_G(U)\oplus K^j_G(V) @>>> K^j_G(U\cup V) \to\\
		@VV \mathcal P\mathcal D V @VV \mathcal P\mathcal D V @VV \mathcal P\mathcal D V\\
		\to RK^{j}_{G}(C_{\tau}(U\cap V)) @>>>RK^{j}_{G}(C_{\tau}(U))\oplus RK^{j}_{G}(C_{\tau}(V)) @>>>RK^{j}_{G}(C_{\tau}(U\cup V)) \to
		\end{CD}
		\]
		%
		where  $j=0, 1$. This diagram commutes. 
		
		Choose a good open cover $\{G\times_{H_i} U_i\}$ for $X$. Using the Five Lemma and induction, it suffices to show that (\ref{eq:PD.map.open}) is an isomorphism when $U$ is homeomorphic to $G\times_{H}\R^m$ for some $m$ and a finite subgroup $H$ of $G$ acting on $\R^m$: that is, show that
		\begin{equation}
		\label{eq:PD.mid}
		\mathcal P\mathcal D:K^{\bullet}_G(G\times_H \R^m)\rightarrow RK^{\bullet}_{G}(C_{\tau}(G\times_H \R^m))
		\end{equation}
		is an isomorphism.
		Recall the following induction isomorphisms for discrete groups:
		\begin{align}
		K_H^{\bullet}(\R^m)&\simeq K_G^{\bullet}(G\times_H\R^m);\\
		RK^{H}_{\bullet}(\R^{2m})&\simeq RK^{G}_{\bullet}(G\times_H\R^{2m}).
		\label{induction}
		\end{align}
		For first isomorphism see~\cite{LO} and for the second see~\cite{BaumHigsonSchick.eq}. The latter isomorphism implies
		\begin{multline*}
		RK^{\bullet}_{H}(C_{\tau}(\R^m))\simeq RK^{H}_{\bullet}(T\R^m)\simeq RK^G_{\bullet}(G\times_H T\R^m)\simeq \\
		RK^G_{\bullet}(T(G\times_H\R^m))\simeq RK_G^{\bullet}(C_{\tau}(G\times_H\R^m)),
		\end{multline*}
		where the third isomorphism follows from $G$ being discrete.
		Thus to prove~(\ref{eq:PD.mid}) it only remains to show that
		\[
		K_H^{\bullet}(\R^m)\simeq RK_{H}^{\bullet}(C_{\tau}(\R^{m})).
		\]
		But because $\R^m$ admits an $H$-Spin$^c$ structure, we have 
		\[
		RK^{H}_{\bullet}(C_{\tau}(\R^{m}))\simeq RK_{H}^{\bullet+m}(\R^{m}),\qquad K^{H}_{\bullet+m}(\R^{m})\simeq K_{H}^{\bullet}(C_{\tau}(\R^{m}))
		\] 
		and so we need only to show 
		$K_H^{\bullet}(C_{\tau}(\R^m))\simeq RK_{H}^{\bullet}(\R^{m})$. This follows from Corollary~4.11 in \cite{Kasparov} or~(\ref{eq:2}), which completes the proof.
	\end{proof}
	
	\begin{remark}
		Our Poincar\'e duality between equivariant $K$-theory and $K$-homology for proper actions of an almost-connected group or a discrete group is optimal in the following sense.
		There are two reasons preventing us from generalising the Poincar\'e duality to a setting beyond discrete groups and almost connected groups. 
		One is that equivariant $K$-theory given by equivariant vector bundles may not be a generalised cohomology theory and so the Mayer-Vietoris sequence cannot be applied.
		The other is that the induction homomorphism~(\ref{induction}) on $K$-homology from an arbitrary compact subgroup $H$ to $G$ is not an isomorphism in general.  
	\end{remark}
	
	\vspace{0.3cm}
	
	\section{Index Theory of $K$-homology Classes}
	\label{sec:Index theory of a K-homology class}
	
	Let $G$ be an almost-connected Lie group acting properly and cocompactly on a manifold $X$.
	Recall that elements of $K^G_0(X)$ are represented by abstract elliptic  operators.
	In this section, we study index theory associated to each element in $K^G_0(X)$ and its relation to induction from a maximal compact subgroup.
	
	\subsection{Higher Index}
	Let $G$ and $X$ be as above and $K$ a maximal compact subgroup of $G$. Let $Y$ be a global $K$-slice of $X$.
	For a proper cocompact action there exists a \emph{cut-off} function, a nonnegative function
	$c\in C_c^{\infty}(X)$ whose integral over every orbit is $1$:
	\[
	\int_Gc(g^{-1}x)dg=1,  \qquad g\in G,\, x\in X. 
	\]
	The function $c$ gives rise to an idempotent $p$ in $C_0(X)\rtimes G$, satisfying
	\begin{equation}
	\label{eq:projection}
	(p(g))(x):=\sqrt{\mu(g^{-1})c(g^{-1}x)c(x)},\qquad g\in G, \,x\in X.
	\end{equation}
	Here $\mu$ is the modular function on $G$; that is, if $dg$ is the left Haar measure on $G$ and $s\in G$, then $\mu$ satisfies $d(gs)=\mu(s)dg.$ 
	Note that the $K$-theory class $[p]$ of $p$ in $K_0(C_0(X)\rtimes G)$ is independent of the choice of $c$.
	
	\begin{definition}[\cite{Kasparov83}]
		The higher index map $\ind_G: K_{\bullet}^G(X)\rightarrow K_{\bullet}(C_r^*(G))$ is given by
		\begin{equation}
		\label{eq:higher.ind}
		\ind_G(x)=[p]\otimes_{C_0(X)\rtimes G} j_r^G(x),
		\end{equation}
		where $[p]\in K_0(C_0(X)\rtimes G)$ and $j^G_r$ is the descent homomorphism 
		\begin{equation}
		\label{eq:descent}
		j^G_r: KK^G_{\bullet}(A, B)\rightarrow KK_{\bullet}(A\rtimes_r G, B\rtimes_r G)
		\end{equation}
		for $A=C_0(X)$ and $B=\C.$
	\end{definition}
	
	\begin{remark}
		\label{rmk:prop.crossed.prod}
		The reduced group $C^*$-algebra $C^*_r(G)$ is isomorphic to the reduced crossed product $\C\rtimes_r G$.
		The higher index map~(\ref{eq:higher.ind}) is defined using the reduced, rather than the full, group $C^*$-algebra, with the idempotent $[p]\in K_0(C_0(X)\rtimes_r G)$.
		However, we have 
		\[
		C_0(X)\rtimes G\simeq C_0(X)\rtimes_r G
		\] 
		whenever $G$ acts properly (cf. {\cite[Theorem~3.13]{Kasparov}}). Thus we shall write $C_0(X)\rtimes G$ in place of $C_0(X)\rtimes_r G$. 
		Note also that since a compact group $K$ acts properly on a point, we have that $C_r^*(K)\simeq C^*(K)$, whereas this does not hold for a general non-compact group $G$. We say that $G$ is {\em amenable} when $C_r^*(G)\simeq C^*(G).$
	\end{remark}
	
	\begin{remark}
		If $X$ is a classifying space for proper actions, $\ind_G$ can be used to define the \emph{analytic assembly map} in the Baum--Connes \cite{Baum-Connes, BCH} and Novikov \cite{Kasparov} conjectures. For a compact group $K$, $K_{0}(C^*_r(K))$ can be identified with the representation ring $R(K)$, while $K_1(C^*_r(K)) = 0$, and $\ind_K$ is the usual equivariant index. 
	\end{remark}
	
	\begin{remark}
		When $G$ is an almost-connected Lie group, a classifying space for proper $G$-actions is $G/K$. Assume that $G/K$ is Spin. For every proper cocompact $G$-space $X,$ 
		there is a continuous $G$-equivariant proper map $p\colon X \to G/K$ (see Theorem \ref{thm:abels}). The map $p_*$ induced on $K$-homology relates the equivariant indices on $X$ and $G/K$ via the diagram
		\begin{equation}
		\label{eq:ind.X.G/K}
		\xymatrix{ K^G_\bullet(X) \ar[r]^-{\ind_G} \ar[d]_-{p_*} & K_{\bullet}(C^*_r(G)). \\
			K^G_\bullet(G/K) \ar[ur]_-{\ind_G}^{\cong} &
		}
		\end{equation}
		Since the Baum--Connes conjecture is true for almost-connected groups by
		Theorem~1.1 in \cite{CEN}, the equivariant index on $G/K$ defines an isomorphism 
		\begin{equation}
		\label{eq:iso.G/K.BC}
		\ind_G : K^G_\bullet(G/K) \cong  K_{\bullet}(C^*_r(G)).
		\end{equation}
	\end{remark}
	
	\subsection{Dirac Induction}
	
	We will assume throughout this subsection and the next that $G/K$ is Spin. We first recall the definition of the {\it Dirac induction} map, which is a special case of a $G$-index map when $X=G/K$. First, note that we have an isomorphism
	\[
	R(K)\simeq K_d^G(G/K),\qquad d=\dim G/K.
	\] 
	If $(\rho, V)$ is an irreducible representation of $K$, its image under this isomorphism is represented by the $K$-homology cycle of the $\Spinc$-Dirac operator $D_{G/K}$ on $G/K$ coupled with $V$. 
	Taking the index of this operator then defines the \emph{Dirac induction} of the class of $(\rho,V)$:
	\begin{equation}
	\label{eq:indG/K1}
	\DInd_K^G: R(K)\rightarrow K_d(C^*_r(G)), \qquad \DInd_K^G([V]):=[p]\otimes_{C_0(G/K)\rtimes G}j_r^G([D_{G/K}^V]).
	\end{equation}
	This map is an isomorphism of abelian groups by the Connes--Kasparov conjecture~\cite{Connes, Kasparov84, CEN}, proved for almost-connected groups in \cite{CEN} based on important earlier results in~\cite{Wassermann, Lafforgue}. 
	
	Equivalently, $\DInd_K^G$ can be constructed as follows. Let 
	\[
	[\partial_{G/K}]\in KK_d^G(C_0(G/K), \C)
	\] 
	denote the Dirac element on $G/K$. Then, as we show below, Dirac induction $R(K)\rightarrow K_d(C^*_r(G))$ can also be defined by
	\begin{equation}
	\label{eq:indG/K2}
	\DInd_K^G([\rho])=[\rho]\otimes_{C_0(G/K)\rtimes G} j_r^G([\partial_{G/K}]), \qquad [\rho]\in R(K)\simeq K_0(C_0(G/K)\rtimes G).
	\end{equation}
	The last isomorphism follows from the fact that $C^*_r(K)$ and $C_0(G/K)\rtimes G$ are Morita-equivalent. To see that definitions (\ref{eq:indG/K1}) and (\ref{eq:indG/K2}) coincide, the following lemma relating $C_r^*(K)$ and $C_0(G/K)\rtimes G$ on the level of $K$-theory is needed. 
	
	\begin{lemma}
		\label{lem:Morita.equivalence}
		Let $V$ be a finitely generated projective module over $C^*_r(K)$. Then $V$ corresponds to the projective $C_0(G/K)\rtimes G$-module $p\cdot (\Gamma(G\times_K V)\rtimes G)$ under the Morita equivalence of $C^*_r(K)$ and $C_0(G/K)\rtimes G.$
		That is, \[
		[V]\mapsto [p\cdot (\Gamma(G\times_K V)\rtimes G)]
		\] 
		in the isomorphism 
		\begin{equation}
		\label{eq:K.ind.K.G/K.G}
		K_0(C^*_r(K))\simeq K_0(C_0(G/K)\rtimes G).
		\end{equation}
	\end{lemma}
	
	\begin{proof}
		From~\cite{EE}, we know that the $C_0(G/K)\rtimes G$-module $p\cdot (C_0(G/K)\rtimes G)$ and the $C^*_r(K)$-module $\C$ correspond under Morita equivalence. In fact, $C_c(G)$ can be regarded as a right $C_c(K)$-module and a left $C_c(G, C_c(G/K))$-module. A cut-off function $c$ on $G/K$ with respect to the $G$-action can be lifted to an element in $C_c(G)$ satisfying 
		\[
		\langle c, c\rangle_{C_c(G, C_c(G/K))}=p,\qquad \langle c, c\rangle_{C_c(K)}=1.
		\] 
		So the projection $1$ in $C^*_r(K)$ and the projection $p$ in $C_0(G/K)\rtimes G$ correspond under Morita equivalence. 
		In fact, this follows from: 
		\[
		\cK(p\cdot (C_0(G/K)\rtimes G))\simeq p\cdot (C_0(G/K)\rtimes G)\cdot p\simeq\C,
		\]
		where $\cK$ is the space of compact operators on the Hilbert $C_0(G/K)\rtimes G$-module.
		See details in~{\cite[Example~5.2]{EE}}.
		Hence, the module $1\cdot C^*_r(K)$ corresponds to $p\cdot (C_0(G/K)\rtimes G)$ under Morita equivalence:
		\[
		1\cdot C^*_r(K)\sim_{M.E.} p\cdot (C_0(G/K)\rtimes G).
		\]
		Here $\cdot$ means module multiplication given by convolution with respect to $K$ or $G$. 
		One directly calculates that $1\cdot C^*_r(K)\simeq\C$, so that we have
		\begin{equation}
		\label{eq:ME}
		\C\sim_{M.E.}p\cdot (C_0(G/K)\rtimes G).
		\end{equation}
		In particular, $a\in\C$ is identified as an element $\tilde a:=p\cdot f_a$ in $p\cdot (C_0(G/K)\rtimes G)$ where $(f_a(g))(x)=a$ for $g\in G$ and $x\in G/K.$
		This means that the Lemma is true when $V=\C$ is the trivial representation.
		The proof for general $V$ then follows by observing that~\eqref{eq:K.ind.K.G/K.G} is an isomorphism of $R(K)$-modules. 
		In fact, \eqref{eq:ME} implies the following module isomorphisms: 
		\begin{equation}
		\label{1}
		p\cdot (\Gamma(G\times_K \C)\rtimes G)\otimes_{C_c(G, C_0(G/K))} C_c(G)\simeq \C;
		\end{equation}
		\begin{equation}
		\label{2}
		C_c(G)\otimes_{C_c(K)}\C\simeq p\cdot (\Gamma(G\times_K \C)\rtimes G).
		\end{equation}
		Applying the $C^*_r(K)$-module $V$ on the left in (\ref{1}) or on the right in (\ref{2}), we obtain versions of the isomorphisms (\ref{1})-(\ref{2}) with $\C$ replaced by $V$. Equivalently, we have: 
		\[
		V\sim_{M.E.}p\cdot(\Gamma(G\times_K V)\rtimes G).
		\]
		Here, every $v\in V$ is identified as an element $\tilde v:=p\cdot f_v$ in $p\cdot (\Gamma(G\times_K V)\rtimes G)$, where 
		\[
		(f_v(g))(hK)=h[(1, v)], \qquad g, h\in G,\, hK\in G/K,\, [(1, v)]\in G\times_K V.
		\]
		This completes the proof.
	\end{proof}
	
	\begin{remark}
		Under the Morita equivalence $C_r^*(K)\simeq C_0(G/K)\rtimes G$, the projection $p$ in $C_0(G/K)\rtimes G$ corresponds the constant function $1$ on $K$. If $\rho_0$ is the trivial representation of $K$, then under the isomorphism 
		\[
		R(K)\simeq K_0(C^*_r(K))\simeq K_0(C_0(G/K)\rtimes G), 
		\]
		we have 
		\[
		R(K)\ni [\rho_0]\leftrightarrow [p]\in K_0(C_0(G/K)\rtimes G).
		\]
	\end{remark}
	
	\begin{remark}
		\label{rmk:Morita.equivalent}
		The isomorphism $K_0(C^*_r(K))\rightarrow K_0(C_0(G/K)\rtimes G)$ given by $[V]\mapsto [p\cdot (\Gamma(G\times_K V)\rtimes G)]$ is the composition of the following maps:
		\begin{align*}
		j: &R(K)\rightarrow KK^G(C_0(G/K), C_0(G/K)) \\
		&[V]\mapsto [(\Gamma(G\times_K V), 0)];\\
		j^G: &KK^G(C_0(G/K), C_0(G/K))\rightarrow KK(C_0(G/K)\rtimes G, C_0(G/K)\rtimes G) \\
		&[(\Gamma(G\times_K V), 0)]\mapsto  [(\Gamma(G\times_K V)\rtimes G, 0)];\\
		[p]\otimes_{C_0(G/K)\rtimes G}: &KK(C_0(G/K)\rtimes G, C_0(G/K)\rtimes G)\rightarrow K_0(C_0(G/K)\rtimes G)\\
		&[(\Gamma(G\times_K V)\rtimes G, 0)]\mapsto [p\cdot(\Gamma(G\times_K V)\rtimes G)].
		\end{align*}
		This description helps us prove the equivalence of definitions (\ref{eq:indG/K1}) and (\ref{eq:indG/K2}).
	\end{remark}
	
	\begin{lemma}
		Let $[V]\in R(K)$. Then definitions (\ref{eq:indG/K1}) and (\ref{eq:indG/K2}) coincide:
		\[
		\ind_G([D_{G/K}^{V}])=[p]\otimes_{C_0(G/K)\rtimes G}j_r^G([D_{G/K}^V])=[V]\otimes_{C_0(G/K)\rtimes G} j_r^G([\partial_{G/K}]).
		\]
	\end{lemma}
	
	\begin{proof}
		Observe by comparing $KK$-cycles that we have
		\[
		j_r^G([D_{G/K}^V])=j^G\circ j([V])\otimes_{C_0(G/K)\rtimes G} j_r^G([\partial_{G/K}]).
		\]
		Thus, \[
		[p]\otimes_{C_0(G/K)\rtimes G} j_r^G([D_{G/K}^V])=[p]\otimes_{C_0(G/K)}j^G\circ j([V])\otimes_{C_0(G/K)\rtimes G} j_r^G([\partial_{G/K}]).
		\]
		The Lemma follows from Remark~\ref{rmk:Morita.equivalent}.
	\end{proof}
	
	\subsection{Higher Index Commutes with Induction}
	
	The main theorem of this section is the following:
	
	\begin{theorem}[Higher index commutes with induction]
		\label{thm:index.comm.induction}
		The following diagram commutes:
		\begin{equation}
		\label{eq:comm.index}
		\begin{CD}
		K^{K}_{\bullet}(Y) @>\ind_K>> K_{\bullet}(C^*_r(K))\\
		@V\KInd_K^G VV @V\DInd_K^GVV\\
		K^{G}_{\bullet+d}(X) @>\ind_G>> K_{\bullet+d}(C^*_r(G)).
		\end{CD}
		\end{equation}
	\end{theorem}
	
	\begin{proof}
		We only need to give the proof for ${\bullet}=0$.
		Consider the $G$-equivariant maps $\lambda: Y\rightarrow pt$, $\tilde\lambda: X=G\times_K Y\rightarrow G/K$ and the contravariant maps on algebras
		\[
		\lambda': \C\rightarrow C_0(Y), \qquad \tilde\lambda': C_0(G/K)\rightarrow C_0(X).
		\] 
		Let $(\cH, f, F)$ be a Kasparov cycle for $K_0^K(Y)$.
		Then the commutativity of 
		\[
		\begin{CD}
		f: C_0(Y)\rightarrow\cL(\cH) @>>> f\circ \lambda'\\
		@VVV @VVV\\
		\tilde f: C_0(X)\rightarrow\cL(G\times_K\cH) @>>> \tilde f\circ\tilde \lambda'=\widetilde{f\circ \lambda'}
		\end{CD}
		\]
		implies the commutativity of 
		\begin{equation}
		\label{eq:3}
		\begin{CD}
		K_0^K(Y) @>\lambda_*>> K^K_0(pt)\\
		@V\simeq VV @V\simeq VV\\
		KK^G(C_0(X), C_0(G/K)) @>\tilde \lambda_*>> KK^G(C_0(G/K), C_0(G/K)).
		\end{CD}
		\end{equation}
		Every element in $K_0^K(pt)\simeq KK^K(\C, \C)$ can be represented by a $K$-vector space $[V]$, or a finite-dimensional representation of $K.$ Regarding $V$ as a $C_r^*(K)$-module, the higher index is the identity map 
		\[
		K_0^K(pt)\rightarrow K_0(C_r^*(K)),\qquad  \ind_K([V])=[V].
		\]
		So in the diagram 
		\begin{equation}
		\label{eq:2}
		\begin{CD}
		KK^K(\C, \C) @>\ind_K>> K_0(C^*_r(K))\\
		@V\simeq VV @V\simeq VV\\
		KK^G(C_0(G/K), C_0(G/K)) @>\ind_G>>K_0(C_0(G/K)\rtimes G),
		\end{CD}
		\end{equation}
		the elements are mapped as
		\[
		\begin{CD}
		[(V, 0)] @>\ind_K>> [(V, 0)]\\
		@VVV  @VVV \\
		[(\Gamma(G\times_K V), 0)] @>\ind_G>>[(p\cdot(\Gamma(G\times_K V)\rtimes G), 0)].
		\end{CD}
		\]
		Thus, diagram~(\ref{eq:2}) commutes if $V$ and $p\cdot(\Gamma(G\times_K V)\rtimes G)$ are Morita-equivalent as modules; but this has been proved in Lemma~\ref{lem:Morita.equivalence}.
		
		Noting that the higher index map factors through the higher index map for classifying spaces (cf.~(\ref{eq:ind.X.G/K})), commutativity of (\ref{eq:3})-(\ref{eq:2}) implies that the following diagram commutes (cf. (\ref{eq:1}) for the left isomorphism):
		\begin{equation}
		\label{eq:4}
		\begin{CD}
		K_0^K(Y) @>\ind_K>> K_0(C^*_r(K))\\
		@V\simeq VV @V\simeq VV\\
		KK^G(C_0(X),C_0(G/K)) @>\ind_G>>K_0(C_0(G/K)\rtimes G).
		\end{CD}
		\end{equation}
		
		Finally, note that for $x\in KK^{G}(C_0(X), C_0(G/K))$, we have
		\begin{align*}
		\ind_G(x)\otimes_{C_0(G/K)\rtimes G} j_r^G([\partial_{G/K}])=&\left([p]\otimes_{C_0(X)\rtimes G} j^G(x)\right)\otimes_{C_0(G/K)\rtimes G} j_r^G([\partial_{G/K}]) \\
		=&[p]\otimes_{C_0(X)\rtimes G} \left(j^G(x)\otimes_{C_0(G/K)\rtimes G} j_r^G([\partial_{G/K}])\right)\\
		=&[p]\otimes_{C_0(X)\rtimes G} j_r^G(x\otimes_{C_0(G/K)} [\partial_{G/K}])\\
		=&\ind_G(x\otimes_{C_0(G/K)} [\partial_{G/K}]).
		\end{align*}
		In other words, the following diagram commutes:
		\begin{equation}
		\label{eq:comm.index.G}
		\begin{CD}
		KK^{G}(C_0(X), C_0(G/K)) @>\ind_G>> K_0(C_0(G/K)\rtimes G))\\
		@VV\otimes_{C_0(G/K)}\,[\partial_{G/K}]V @VV\otimes_{C_0(G/K)\rtimes G} \,j_r^G([\partial_{G/K}])V\\
		K^{G}_d(X) @>\ind_G>> K_d(C^*_r(G)).
		\end{CD}
		\end{equation}
		Putting together (\ref{eq:4}) and (\ref{eq:comm.index.G}) completes the proof.
	\end{proof}

	\begin{remark}
		The commutative diagram in Theorem~\ref{thm:index.comm.induction} is closely related to the \emph{quantisation commutes with induction} results of \cite{HochsDS, HochsPS}.
		We acknowledge that Theorem~\ref{thm:index.comm.induction} was proved in \cite{HochsDS} by more involved diagram chasing. Here we have given a more direct proof by making use of the properties of $KK$-theory.
	\end{remark}
	
	We end this subsection by stating an implication of Remark~\ref{sec:Consequence} and Theorem~\ref{thm:index.comm.induction}. 
	
	\begin{corollary}
		\label{prop quant ind}
		For every $x\in K_{\bullet}^G(X)$, there exists a geometric cycle $(M, E, f)$ representing an element in $K_{\bullet}^{G, geo}(X)$ such that $x=f_*([D_{M, E}])$, and upon choosing a $K$-slice $N$ of $M$ with a compatible $K$-equivariant $\Spinc$-structure, we have 
		\[
		\ind_G x=\ind_G f_*([D_{M, E}])=\DInd_K^G(\ind_K (f|_N)_*([D_{N, E|_N}]))) \in K_{\bullet}(C^*_{r}(G)).
		\]
	\end{corollary}
	
	\vspace{0.3cm}

	\section{A Trace Theorem for Equivariant Index}\label{sec:reduction}
	
	We now focus on the case when $G$ is a connected, semisimple Lie group with finite centre and $K$ a maximal compact subgroup, such that $G/K$ is Spin. We further assume that $d = \dim G/K$ is even and rank $G$ = rank $K$, in order to allow $G$ to have discrete series representations. Let $M$ be a $G$-equivariant $\Spinc$-manifold with Dirac operator $D_M$ and $N$ a $K$-slice of $G$ with Dirac operator $D_N$ associated with a compatible (see Remark \ref{rem:Spinc.stru.ind}) $K$-equivariant $\Spinc$-structure on $N$. We shall use the commutativity result in Section 4 together with known representation-theoretic properties of the group $K_0(C_r^*(G))$ for such $G$ to deduce a formula for the $L^2$-index of $D_M$ in terms of the $L^2$-index of $D_N$. Since twisted Spin$^c$-Dirac operators exhaust $K_0^G(M)$, this allows us to relate the $L^2$-index of any operator representing a class in $K_0^G(M)$ to the $L^2$-index of $K$-equivariant Dirac operators in $K_0^K(N)$.
	
	\begin{remark}
		For simplicity, we have assumed that $M$ has a $G$-equivariant Spin$^c$-structure. Of course, since the commutativity result of Section 4 did not require this, all of the statements in this section apply in an appropriate form to operators on any proper $G$-cocompact manifold $X$, not necessarily having a Spin$^c$-structure, and a compact $K$-slice $Y$. 
	\end{remark}
	
	\noindent Let $\mathcal{R}_G\subseteq\mathcal{B}(L^2(G))$ be the commutant of the right-regular representation $R$ of $G$. Then $\mathcal{R}_G$ is a von Neumann algebra with a faithful, normal, semi-finite trace $\tau$ determined by
	$$\tau(R(f)^*R(f)) = \int_G |f(g)|^2\,dg,$$
	
	\noindent for all $f\in L^2(G)$ such that $R(f)$ is a bounded operator on $L^2(G)$. Let $E$ be a $\mathbb{Z}_2$-graded $G$-equivariant vector bundle over $M$, and let $E\cong G\times_K (E|_N)$ be a fixed choice of decomposition. Denote by $\mathcal{R}_G(E)$ the $K$-invariant elements of 
$$\mathcal{R}_G\otimes\mathcal{B}(L^2(N,E|_N)).$$
Then $\mathcal{R}_G(E)$ is equipped with a natural trace, which we will also denote by $\tau$, given by combining $\tau$ on $\mathcal{R}_G$ with the trace on bounded operators. 
	
	Let $D=D^+\oplus D^-$ be an odd-graded first-order $G$-invariant elliptic operator acting on $E$ and defining a class in $K_0^G(M)$. By definition, the {\it $L^2$-index} of $D$ is (see Connes-Moscovici~\cite{CM82} or Wang \cite{Wang})
	$$\L2ind(D) := \tau(\textnormal{pr}_{\textnormal{ker}D^+}) - \tau(\textnormal{pr}_{\textnormal{ker}D^-}),$$
	
	\noindent where pr$_L$ denotes the projection of $L^2(M,E)$ onto a $G$-invariant subspace $L$. It was shown in \cite{Wang} that the $\L2ind$ of $D$ can be calculated via the composition of maps
	$$K_0^G(M)\xrightarrow{\ind_G}K_0(C_r^*(G))\xrightarrow{\tau_G}\mathbb{R},$$ 
	
	\noindent where $\tau_G$ is the von Neumann trace on $C_r^*(G)$, defined first on the dense subspace $C_c(G)$ by evaluation at the identity $e\in G$:
	$$\tau_G: C_c(G)\rightarrow\mathbb{C},\qquad f\mapsto f(e),$$
	
	\noindent and extended to $C_r^*(G)$. The same paper provides an integral formula for the $L^2$-index:
\begin{equation}\label{eq:L2indexformula}
	\L2ind(D) = \int_{TM} (c\circ\pi)(\hat{A}(M))^2\cup\textnormal{ch}(\sigma_D),
\end{equation}
	\noindent where $\pi:TM\rightarrow M$ is the natural projection and $c$ is a compactly supported cut-off function on $M$, with the property that $\int_G c(g^{-1}x)\, dg = 1$ for any $x\in M$. Since we have assumed that the manifold $M$ is Spin$^c$, this integral is equal to
	$$\int_M c(x) e^{\frac{1}{2}c_1(L_M)}\hat{A}(M),$$
	
	\noindent where $L_M$ is the line bundle defining the $G$-equivariant Spin$^c$-structure on $M$. 
	
	On the other hand, it is known that, for our chosen class of Lie groups $G$, the discrete series representations of $G$ can be identified with elements of $K_0(C_r^*(G))$ (although not necessarily all of them) in the following way. Let $(H,\pi)$ be a discrete series representation of $G$. Let $d_H$ be the {\em formal degree} of $(H,\pi)$ - equal to its mass in the Plancherel measure - and $c_v$ be a {\it matrix coefficient} of the representation, given by
	$$c_v:G\rightarrow\mathbb{C},\qquad c_v(g):=\langle v,\pi(g)v\rangle,$$ 
	for some $v\in H, ||v|| = 1$ (it is independent of the choice of $v$). It can be verified that $d_H c_v$ is a projection in the $C^*$-algebra $C_r^*(G)$ and hence defines an element in its $K$-theory. It is also known that the von Neumann trace $\tau_G$ (defined above) of the class $[d_H c_v]\in K_0(C_r^*(G))$ is equal to $(-1)^{d/2}$ times the formal degree (see \cite{LafforgueICM}). 

	For a more extensive discussion we refer to \cite{LafforgueICM}. It is shown there that the map taking $(H,\pi)$ to $[d_H c_v]\in K_0(C_r^*(G))$ is an injection from the set of discrete series representations to a set of generating elements for $K_0(C_r^*(G))$ \cite{LafforgueICM}. 

	It is interesting to note that for connected semisimple Lie groups having discrete series, including $G=SL(2, \R)$, the elements of $K_0(C_r^*(G))$ can be given {\it entirely}, in the aforesaid way, as representations from the discrete series and limits of discrete series of $G$. See~\cite{Wassermann}. This motivates us to understand $K_0(C_r^*(G))$, as well as $G$-equivariant index theory, from the point of view of representation theory of $G$. By comparison, the theory of representations of compact groups already plays an important role in index theory. The goal of this section is to observe an explicit representation-theoretic relationship between a $G$-equivariant index and its corresponding $K$-equivariant index, in the sense of the previous section.
	
	The elements of $K_0(C_r^*(G))$, considered as representations of $G$, are in bijection with the irreducible representations of $K$ via the Dirac induction map $\DInd$ (see \cite{BCH}). On the other hand, there exists a formula for the formal degree of a discrete series representation of $G$ in terms of the highest weight of a corresponding $K$-representation together with information about the root systems of $G$ and $K$ (see \cite{AtiyahSchmid} or \cite{Lafforgue}). 
	
	To state this formula precisely, let us first fix the following notation. Let $T$ be a common maximal torus of $K$ and $G$, and choose a Weyl chamber $C$ for the root sytem $\Phi$ of $\mathfrak{g}$, determining a set $\Phi^+$ of positive roots. Pick the Weyl chamber for $\mathfrak{k}$ that contains $C$, and let $\Phi^+_c$ be the system of compact roots. Let $\Phi^+_n:=\Phi^+\backslash\Phi^+_c$, $\rho_c$ and $\rho_n$ be the half-sums of the positive-compact and non-compact roots respectively, and let $\rho:=\rho_c+\rho_n$.
	
	We can relate the formula for the formal degree proved in \cite{LafforgueICM} to the commutativity result Section 4, as follows. This gives an interpretation of a result of Atiyah and Schmid (p. 25 of \cite{AtiyahSchmid}) in terms of equivariant index theory.
	
	\begin{proposition}\label{prop:trace1}
		Let $G$ and $K$ be as above, with $G/K$ Spin. Let $\mu$ be the highest weight of an irreducible representation $V_\mu$ of $K$. Suppose $\mu+\rho_c$ is the Harish-Chandra parameter of a discrete series representation of $(H,\pi)$ of $G$. Then the formal degree of $(H,\pi)$ is given by
		\begin{align*}
		d_H &= (-1)^{d/2}\,\tau_G([\DInd([V_\mu])])=\prod_{\alpha\in\Phi^+}\frac{(\mu+\rho_c,\alpha)}{(\rho,\alpha)},
		\end{align*}
		with both sides vanishing when the class $\DInd[V_{\mu}]\in K_0(C^*_r(G))$ is not given by a discrete series representation. We have a well-defined commutative diagram:
		\begin{equation*}\label{tracediagram}
		\begin{tikzcd}
		K_0^G(M)\arrow{r}{\ind_G} & K_0(C_r^*(G))\arrow{rd}{\tau_G} & \ \\ 
		\ & \ & \mathbb{R},\\
		K_0^K(N)\arrow{uu}{\KInd}\arrow{r}{\ind_K} & R(K)\arrow{uu}{\DInd}\arrow{ru}{\Pi_K} & \
		\end{tikzcd}
		\end{equation*}
		where $\Pi_K([V_\mu]):=(-1)^{d/2}\,\prod_{\alpha\in\Phi^+}\frac{(\mu+\rho_c,\alpha)}{(\rho,\alpha)}.$ Here the Haar measure on $G$ is normalised by
		$$\textnormal{vol }K = \textnormal{vol }M_1/K_1 = 1,$$
		where $M_1$ is a maximal compact subgroup of the universal complexification $G^\mathbb{C}$ of $G$ and $K_1<G_1$ a maximal compact subgroup of a real form $G_1$ of $G^\mathbb{C}$ (see \cite{AtiyahSchmid} for more details).
	\end{proposition}
	
	\begin{remark}
		We have $\tau_G(\DInd[V_{\mu}])=0$ when $\DInd[V_{\mu}]\in K_0(C^*_r(G))$ is not given by a discrete series representation of $G$, since the Harish-Chandra parameter is then singular, hence the right-hand side of the formula above vanishes.
	\end{remark}
	
	We can state this result in a more suggestive way that makes clear the relationship between the von Neumann traces $\tau_G$ on $K_0(C_r^*(G))$ and $\tau_K$ on $R(K)\cong K_0(C_r^*(K))$.
	
	\begin{proposition}\label{prop:trace2}
		With the notation above, we have:
		
		$$\tau_G([\DInd([V_\mu])]) = (-1)^{d/2}\left(\frac{\prod_{\alpha\in\Phi_n^+}(\mu+\rho_c,\alpha)\prod_{\alpha\in\Phi_c^+}(\rho_c,\alpha)}{\prod_{\alpha\in\Phi^+}(\rho,\alpha)}\right) \tau_K([V_\mu]).$$
	\end{proposition}
	\begin{proof}
		By the Weyl dimension formula for $V_\mu$ and the Proposition above,

\begin{align*}
\tau_G([\DInd([V_\mu])]) &= (-1)^{d/2}\,d_H\\
&= (-1)^{d/2}\, \prod_{\alpha\in\Phi^+}\frac{(\mu+\rho_c,\alpha)}{(\rho,\alpha)}\\
&=(-1)^{d/2}\,\prod_{\alpha\in\Phi^+_c}\frac{(\mu+\rho_c,\alpha)}{(\rho_c,\alpha)}\left(\frac{\prod_{\alpha\in\Phi_n^+}(\mu+\rho_c,\alpha)\prod_{\alpha\in\Phi_c^+}(\rho_c,\alpha)}{\prod_{\alpha\in\Phi^+}(\rho,\alpha)}\right)\\
&=(-1)^{d/2}\,\dim V_\mu\left(\frac{\prod_{\alpha\in\Phi_n^+}(\mu+\rho_c,\alpha)\prod_{\alpha\in\Phi_c^+}(\rho_c,\alpha)}{\prod_{\alpha\in\Phi^+}(\rho,\alpha)}\right).\\
\end{align*}
		It can be verified that the trace $\tau_K$ applied to a general element $[V]\in R(K)\cong K_0(C_r^*(K))$ returns the dimension of the representation space $V$, which concludes the proof.
	\end{proof}
	\noindent Thus $\tau_G$ and $\tau_K$ are related by a scalar factor that, up to a sign, depends only upon the highest weight $\mu$ of the representation $[V_\mu]\in R(K)$ and the chosen root systems of $\mathfrak{g}$ and $\mathfrak{k}$. In particular, as one sees by direct computation in the example $G = SL(2,\mathbb{R})$ and $K=SO(2)$, this scalar factor varies significantly depending on the weight $\mu$ (see \cite{BCH} for more details on the calculation in this case), reflecting the fact that Dirac induction plays a significant role in relating the two $L^2$-indices for $G$ and $K$.
	
	Finally, as was mentioned in earlier, the von Neumann traces $\tau_G$ and $\tau_K$ give rise to $L^2$-indices when applied to the $G$ and $K$-equivariant indices of operators on $M$ and $N$ respectively. Using Wang's formula for the $L^2$-index (\ref{eq:L2indexformula}), the result above becomes an equality of integrals involving characteristic classes on the non-compact manifold $M$ and the compact manifold $N$, as follows.
	
	\begin{corollary}
		Let $D_M$ and $D_N$ be Spin$^c$-Dirac operators on $M$ and $N$ for compatible equivariant Spin$^c$-structures, which are defined by line bundles $L_M$ and $L_N$ respectively, where $\ind_G(D_M)\in K_0(C^*_r(G))$ corresponds to a discrete series representation with Harish-Chandra parameter $\mu+\rho_c.$ Then the $L^2$-indices of $D_M$ and $D_N$ are related by
		
		\begin{align*}
		\L2ind(D_M)&=\int_M c(x)e^{\frac{1}{2}c_1(L_M)}\hat{A}(M)\\
		&=(-1)^{d/2}\left(\frac{\prod_{\alpha\in\Phi_n^+}(\mu+\rho_c,\alpha)\prod_{\alpha\in\Phi_c^+}(\rho_c,\alpha)}{\prod_{\alpha\in\Phi^+}(\rho,\alpha)}\right) \int_N e^{\frac{1}{2}c_1(L_N)}\hat{A}(N)\\
		&=(-1)^{d/2}\,\left(\frac{\prod_{\alpha\in\Phi_n^+}(\mu+\rho_c,\alpha)\prod_{\alpha\in\Phi_c^+}(\rho_c,\alpha)}{\prod_{\alpha\in\Phi^+}(\rho,\alpha)}\right)\L2ind(D_N).\\
		\end{align*}
	\end{corollary}

	\section{Positive Scalar Curvature for Proper Cocompact Actions}\label{sec:psc}
	We now apply equivariant index theory to study obstructions to the existence of invariant metrics of positive scalar curvature on proper cocompact manifolds, before studying the existence of such metrics in the next subsection. 
	\subsection{Obstructions} Suppose that $G$ is a non-compact Lie group acting properly and cocompactly on a smooth, $G$-equivariantly spin, complete Riemannian manifold $M$. Suppose that the dimension of $M$ is even. By the Lichnerowicz formula \cite{Lichnerowicz} for the square of the $G$-Spin-Dirac operator $\dirac$,
	$$
	{\dirac}^2 = (\nabla^S)^\dagger \nabla^S + \kappa_M/4 
	$$
	where $\kappa_M$ denotes the scalar curvature of $M$ and $\nabla^S$ denotes the connection on the spinor bundle that is induced by the Levi-Civita connection on $M$. 
	Suppose now that $M$ has pointwise-positive scalar curvature. Since $M/G$ is compact, we have $\inf(\kappa_M)>0$, and thus ${\dirac}^2$ is a strictly positive operator from the Sobolev space $L^{2,2}(S)$ to the Hilbert space $L^2(S)$, with a bounded inverse \[
	(\dirac^2)^{-1}: L^2(S)\rightarrow L^{2,2}(S)
	\] 
	(see Theorem~2.11 in \cite{Gromov-Lawson}).
	Let $C^*(G)$ be the maximal group $C^*$-algebra of $G$. The space $C_c^\infty(S)$ of smooth cocompactly supported sections of $S$ has the structure of a right pre-Hilbert module over $C_c(G)$ prescribed by the formulas (for more details see \cite{Kasparov2016} Section 5):
	\begin{align*}
	(s\cdot b)(x) &= \int_G g(s)(x)\cdot b(g^{-1})\,dg\in C_c^\infty(S),\\
	(s_1,s_2)(g)&=\int_M(s_1(x),g(s_2)(x))\,d\mu\in C_c^\infty(G),
	\end{align*}
	for $s,s_1,s_2\in C_c^\infty(S)$ and $b\in C_c^\infty(G)$. Let us denote by $\mathcal{E}$ the Hilbert $C^*(G)$-module completion $\mathcal{E}$ of $C_c^\infty(S)$. Theorem~5.8 of \cite{Kasparov2016} shows that the pseudodifferential operator  $F=\dirac(\dirac^2+1)^{-\frac12}: \Gamma(S)\rightarrow\Gamma(S)$ defines an element, which we will also call $F$, in the bounded adjointable operators $\mathcal{L}(\mathcal{E})$, with an index in $K_0(C^*(G))$.
	
	The main result of this section is the following theorem on the vanishing of the equivariant index.
	
	\begin{theorem}[Vanishing Theorem 1]\label{vanishingtheorem1}
		Let $M, G, S$ and $\dirac$ be as above, with $M$ being even-dimensional. Suppose $M$ admits a Riemannian metric of pointwise-positive scalar curvature. Then
		$$
		{\rm index}_G(\dirac) = 0 \in K_0(C^*(G)).
		$$
	\end{theorem}
	\begin{proof}
		The operator $\dirac : C_c^\infty(S)\rightarrow C_c^\infty(S)$ gives a densely-defined (and a priori unbounded) operator on the Hilbert $C^*(G)$-module $\mathcal{E}$. By $G$-invariance, one can verify that $\dirac$ is symmetric: that is, we have $C_c^\infty(S)\subseteq\textnormal{dom}(\dirac^*)$, where
		$$\textnormal{dom}(\dirac^*):=\{ y\in\mathcal{E}:\exists z\in\mathcal{E}\textnormal{ with }\langle\dirac x,y\rangle_\mathcal{E} = \langle x,z\rangle_\mathcal{E}\,\,\,\forall x\in C_c^\infty(S)\}.$$
		Hence $\dirac$ is a closable operator, and we shall write $\rm{dom}(\dirac)$ for the domain of its closure, a closed Hilbert $C^*(G)$-submodule of $\mathcal{E}$. In what follows we will be concerned with the closed operator $\dirac:\rm{dom}(\dirac)\rightarrow\mathcal{E}$. Define the {\em graph norm} $||\cdot||_{\textnormal{dom}(\dirac)}$ on $\textnormal{dom}(\dirac)$ to be the norm induced by the $C^*(G)$-valued inner product
		$$\langle u,v \rangle_{\textnormal{dom}(\dirac)} := \langle u,v\rangle_\mathcal{E} + \langle \dirac u, \dirac v\rangle_\mathcal{E}.$$
		With this norm, $\dirac$ is a bounded adjointable operator $\textnormal{dom}(\dirac)\rightarrow\mathcal{E}$. Since $\dirac$ is a first-order elliptic differential operator on a $G$-cocompact manifold, Theorem~5.8 of \cite{Kasparov2016} implies that both operators $\dirac\pm i$ have dense range as operators on $\mathcal{E}$. Thus $\dirac$ is a self-adjoint regular operator on the Hilbert $C^*(G)$-module $\mathcal{E}$ \cite{Lance}.
		
		Our main task is to show that $\dirac$ has an inverse $\dirac^{-1}:\mathcal{E}\rightarrow\textnormal{dom}(\dirac)$ that is bounded and adjointable. First we show that its square $\dirac^2:\textnormal{dom}(\dirac^2)\rightarrow\mathcal{E}$ is invertible in the bounded adjointable sense. Here the domain 
		$$\textnormal{dom}(\dirac^2):=\{u\in\textnormal{dom}(\dirac):\dirac u\in\textnormal{dom}(\dirac)\}$$
		is equipped with the graph norm induced by the inner product 
		$$\langle u,v\rangle_{\textnormal{dom}(\dirac^2)}:=\langle u,v\rangle_\mathcal{E} + \langle\dirac u,\dirac v\rangle_\mathcal{E} + \langle\dirac^2 u,\dirac^2 v\rangle_\mathcal{E}.$$
		By Proposition 1.20 of \cite{Ebert}, $\dirac^2$ is a densely-defined, closed, self-adjoint regular operator on $\mathcal{E}$. We proceed along the lines of \cite{Cecchini} Proposition 4.9. By regularity, $\dirac^2+\mu^2$ is surjective for every positive number $\mu^2$ (see Chapter 9 of \cite{Lance}). Further, since $\dirac^2$ is strictly positive, $\dirac^2 + \mu^2$ is injective. By the open mapping theorem, its inverse $(\dirac^2+\mu^2)^{-1}$ is bounded. It remains to show that $(\dirac^2+\mu^2)^{-1}$ is adjointable. Write for short $B:=(\dirac^2+\mu^2)^{-1}$. Note that $B$ is self-adjoint as a bounded operator $\mathcal{E}\rightarrow\mathcal{E}$ (defined by composing $B$ with the bounded inclusion $\textnormal{dom}(\dirac^2)\hookrightarrow\mathcal{E}$), which follows from Lemma 4.1 in \cite{Lance} and the estimate
		$$\langle u,Bu\rangle_{\mathcal{E}} = \langle(\dirac+\mu^2)Bu,Bu\rangle_\mathcal{E}\geq\mu^2\langle Bu,Bu\rangle_\mathcal{E}\geq 0.$$
		Next, for any $w\in\mathcal{E}$ and $u\in\textnormal{dom}(\dirac^2)$, we have
		\begin{align*}
		\langle B u,w\rangle_{\textnormal{dom}(\dirac^2)} &= \langle\dirac^2 Bu,\dirac^2 w\rangle_\mathcal{E} +\langle\dirac Bu,\dirac w\rangle_\mathcal{E} + \langle B u, w\rangle_{\mathcal{E}}\\
		&=\langle(\dirac^2+\mu^2)Bu,\dirac^2 w\rangle_{\mathcal{E}} + (1-\mu^2)\langle Bu,\dirac^2 w\rangle_{\mathcal{E}}+ \langle u,Bw\rangle_{\mathcal{E}}\\
		&=\langle u,\dirac^2 w\rangle_{\mathcal{E}} + (1-\mu^2)\langle u,B\dirac^2 w\rangle_{\mathcal{E}} + \langle u,B w\rangle_{\mathcal{E}}\\
		&= \langle u,(\dirac^2 + (1- \mu^2)B\dirac^2 + B)w\rangle_{\mathcal{E}},
		\end{align*}
		where we used symmetry of $\dirac$ in $\mathcal{E}$ self-adjointness of $B$ as shown above. This shows that $(\dirac^2+\mu^2)^{-1}\in\mathcal{L}(\mathcal{E},\textnormal{dom}(\dirac^2))$. 
		
		We claim that $\dirac^2$ is invertible. For write it as $(1-\mu^2(\dirac^2+\mu^2)^{-1})(\dirac^2+\mu^2)$. Since $\dirac^2$ is a strictly positive operator, there exists $C>0$ such that for all $s\in\textnormal{dom}(\dirac^2)$, we have $\langle\dirac^2 s,s\rangle_{\mathcal{E}}\geq C\langle s,s\rangle_{\mathcal{E}}$. It follows from the Cauchy-Schwarz inequality for Hilbert modules that for any element $t\in\mathcal{E}$,
		$$||\mu^2(\dirac^2+\mu^2)^{-1}t||_\mathcal{E}\leq\frac{\mu^2}{\mu^2+C}||(\dirac^2+\mu^2)Bt||_\mathcal{E} = \frac{\mu^2}{\mu^2+C}||t||_\mathcal{E}.$$
		Hence $(1-\mu^2(\dirac^2+\mu^2)^{-1})$ has an adjointable inverse given by a Neumann series. It follows that $\dirac^2$ has a bounded adjointable inverse 
		$$(\dirac^2)^{-1} = B(1-\mu^2(\dirac+\mu^2)^{-1})^{-1}:\mathcal{E}\rightarrow\textnormal{dom}(\dirac^2).$$
		One verifies that $\dirac(\dirac^2)^{-1}:\mathcal{E}\rightarrow\textnormal{dom}(\dirac)$ is a two-sided inverse for $\dirac$. Thus $\dirac+\lambda$ is invertible for all $\lambda$ in a ball around $0\in\mathbb{C}$. Thus, by the functional calculus for self-adjoint regular operators (see Theorem 1.19 of \cite{Ebert} for a list of its properties), we may define the operator $$\frac{\sqrt{\dirac^2+1}}{\dirac}\in\mathcal{L}(\mathcal{E}),$$ which is the inverse of $F=\frac{\dirac}{\sqrt{\dirac^2+1}}\in\mathcal{L}(\mathcal{E})$ by the homomorphism property of the functional calculus. Recall that the $G$-equivariant index of $\dirac$ can be computed by taking trace of the idempotent element
		\[
		\begin{bmatrix}S_0^2 & S_0(1+S_0)Q \\ S_1F & 1-S_1^2\end{bmatrix}-\begin{bmatrix}0 & 0 \\ 0 & 1 \end{bmatrix}\in\textnormal{Mat}_{\infty}(\mathcal{K}(\mathcal{E})),
		\] 
		where $Q$ is any parametrix of $F$ and $S_0=1-QF$, $S_1=1-FQ$ (see for example the first display on p.~353 of \cite{Connes-Moscovici}).
		In our setting we can take $Q = F^{-1}$, hence $S_0=S_1=0$. It then follows that the $G$-equivariant index of $\dirac$ vanishes, after making the canonical identification $K_0(\mathcal{K}(\mathcal{E}))\cong K_0(C^*(G))$.
	\end{proof}
	Now, as defined in \cite{MZ}, integration $\int_G$ on $L^1(G)$ induces a tracial map on $C^*(G)$ and also on its $K$-theory. By Appendix \ref{sec:AppA} of this paper, 
	we have $\int_G({\rm index}_G(\dirac)) = {\rm ind}_G(\dirac)$ where the right-hand side is the $G$-invariant index of Mathai-Zhang \cite{MZ}. Thus we immediately obtain as a corollary
	the following recent result of Zhang \cite{Zhang17}:
	
	\begin{corollary}[\cite{Zhang17}]
		Let $M, G$ and $\dirac$ be as in Theorem \ref{vanishingtheorem1}, with $M$ being even-dimensional. Then
		$${\rm ind}_G(\dirac)=0,$$
		where ${\rm ind}_G$ denotes the Mathai-Zhang index.
	\end{corollary} 
	
	Using the canonical projection from $C^*(G)$ to $C^*_r(G)$, Theorem \ref{vanishingtheorem1} also gives a vanishing result for the reduced equivariant index:
	\begin{theorem}[Vanishing Theorem 2]
		Let $M, G$ and $\dirac$ be as in Theorem \ref{vanishingtheorem1}, with $M$ being even-dimensional. Then 
		$$
		{\rm index}_G(\dirac) = 0 \in K_0(C^*_r(G)).
		$$
	\end{theorem}
	
	Upon applying the von Neumann trace $\tau$ to this index, we conclude that 
	$${\rm index}_{L^2}(\dirac) = \tau({\rm index}_G(\dirac)) = 0,$$ where the $L^2$-index on the left is 
	as defined in \cite{Wang}. Let $c\in C^{\infty}_c(M)$ be a cut-off function, that is to say a non-negative function satisfying 
	\[
	\int_Gc(g^{-1}m)\, dg=1
	\]
	for all $m\in M$, for a fixed left Haar measure $dg$ on $G$. The \emph{averaged $\hat A$-genus} of the action by $G$ on $M$ is then defined to be (see also \cite{Fukumoto})
	\[
	\Aav(M) := \int_M c\hat A_M.
	\]
	We have as a consequence of Vanishing Theorem 2, together with the $L^2$-index theorem of Wang \cite{Wang}, the following result proved by Fukumoto (Corollary B in \cite{Fukumoto}):
	
	\begin{corollary}[\cite{Fukumoto}]
		Let $M$ and $G$ be as in Theorem \ref{vanishingtheorem1}, and assume further that $G$ is unimodular. Then $ \Aav(M)=0$.
	\end{corollary}
	
	\subsection{Existence}
	The previous obstruction results are motivated by the following existence results. In this subsection, we suppose that $G$ is almost-connected, while $M$ is still a proper $G$-cocompact manifold. By Abels' slice theorem \cite{Abels}, $M$ is $G$-equivariantly diffeomorphic to $G\times_K N$,
	where $K$ is a maximal compact subgroup of $G$, $N$ is a compact manifold with an action of $K$. One of our main theorems is the following.
	
	\begin{theorem}\label{thm:induction}
		Let $G$ be an almost-connected Lie group and that $K$ is a maximal compact subgroup of $G$. If $N$ is a compact manifold with a $K$-invariant Riemannian metric of positive scalar curvature, then  $M=G\times_K N$ has a $G$-invariant Riemannian metric of pointwise-positive scalar curvature.
	\end{theorem}
	
	To prove it, recall the following theorem of Vilms: 
	
	\begin{theorem}[\cite{Vilms}]
		Let $\pi\colon M\to B$ be a fibre bundle with fibre $N$ and structure group $K$. 
		Let $g_B$ be a Riemannian metric on $B$ and $g_N$ be a $K$-invariant Riemannian metric on $N$. Then there is a Riemannian metric $g_M$ on $M$ such that $\pi$ is 
		a Riemannian submersion with totally geodesic fibres.
	\end{theorem}
	
	We prove an equivariant version this result, namely:
	
	\begin{theorem}
		\label{thm:eq.Vilms}
		Let $\pi\colon M\to B$ be a fibre bundle with compact fibre $N$ and structure group $K$. 
		Suppose that $M$ and $B$ both have proper $G$-actions making $\pi$ $G$-equivariant.
		Let $g_N$ be a $K$-invariant Riemannian metric on $N$. Then there is a $G$-invariant Riemannian metric $g_M$ on $M$ such that $\pi$ is 
		a $G$-equivariant Riemannian submersion with totally geodesic fibres.
	\end{theorem}
	\begin{proof}
		We adapt the proof of Theorem 3.5 in \cite{Vilms}. Let us choose an Ehresmann connection $H$ on $M$ (with the horizontal lifting property), so that $TM = H\oplus V$, where $V$ is the vertical subbundle $\ker\pi_*$. By a result of Palais \cite{Palais} there exists a $G$-invariant metric $g_B$ on $B$. Using the map $\pi_*|_H$, this induces a $G$-invariant metric $g_H$ on $H$. Also $g_N$ is $K$-invariant and so can be transferred to a metric $g_V$ on $V$. Define the metric $g_M = \langle\,,\rangle_M$ as the direct sum $g_H\oplus g_V$. One verifies that $\pi$ is a Riemannian submersion. In what follows, let the Levi-Civita connections on $M$ and $N$ be denoted $D$ and $D^N$ respectively.
		
		Note that parallel transport along a curve in $B$ gives a $K$-isomorphism of fibres of $M$. Since the metric on $N$ is $K$-invariant, the above construction of $g_M$ means that parallel transport is furthermore an isometry of the fibres.
		
		To show that the fibres of $M$ are totally geodesic, it suffices to show that $D$ and $D^N$ act in the same way on tangent vectors to curves in any given fibre. Thus take any curve $\alpha: I\rightarrow M_b$, where $M_b$ is the fibre over a base point $b\in B$, and parameterise it proportionally to arc-length. Let $v(0)$ be a given horizontal vector at $\alpha(0)$, and define $\sigma(t,s)$ to be the parallel translation of $\alpha(t)$ along some curve in $B$ with initial vector $\pi_* v(0)$. Define the horizontal vector field
		$$v(t):=\frac{\partial}{\partial s}\sigma(t,0)$$
		along $\alpha(t)$. For each $s$, let $f(s)$ denote the arc-length of the curve $\sigma(s,t)$ from $t=0$ to $1$. Since parallel transport is an isometry, the arc-length $f(s)$ is a constant with respect to $s$. In particular, we have $f'(0) = 0$. By the formula for the first variation of the arc-length \cite{Hermann}, we have
		$$f'(0) = \int_0^1 \frac{1}{||\alpha'(0)||}\langle D_t\alpha'(t),v(t)\rangle_M\, dt = \int_0^1 \frac{1}{f(0)}\langle D_t\alpha'(t),v(t)\rangle_M\, dt = 0,$$
		whence by continuity, there must be $t_1\in I$ such that $\langle D_t\alpha'(t_1),v(t_1)\rangle_M = 0$. Next let $\alpha_1(t):=\alpha(\frac{t}{2})$, $0\leq t\leq 1$, and apply the same procedure. Since $D_t \alpha'_1(t)$ and $D_t \alpha'(t)$ are proportional, we have $\langle D_t\alpha'(t_2),v(t_2)\rangle_M = 0$ for some $0\leq t_2\leq\frac{1}{2}$. Repeating this procedure gives a sequence $t_i\rightarrow 0$ for which the quantity $\langle D_t\alpha'(t),v(t)\rangle_M$ vanishes, which implies that $\langle D_t\alpha'(0),v(0)\rangle_M = 0$.
		
		As $v(0)$ was arbitrarily chosen in $H_b$, we conclude that $D_t\alpha'(0)$ is vertical and so equal to $D_t^N\alpha'(0)$. This concludes the proof that the fibres are totally geodesic.
	\end{proof}
	
	\begin{proof}[Proof of Theorem \ref{thm:induction}] Let $\kappa_{G/K}$ denote the scalar curvature of the  $G$-invariant  Riemannian metric $g_{G/K}$ on the base. Note that since $G/K$ is a homogeneous space, $-\infty<\kappa_{G/K}<0$ is a negative constant. Let $H\subseteq TM$ be an Ehresmann connection. Then as in the proof of Theorem \ref{thm:eq.Vilms} above, we may lift $g_{G/K}$ to a $G$-invariant metric $g_H$ on $H$, as well as lift the $K$-invariant Riemannian metric $g_N$ on $N$ to a metric on the vertical subbundle $V\subseteq TM$. Define a $G$-invariant metric on $M$ by $g_M := g_H\oplus g_V$. 
		
		Since $N$ is compact by hypothesis, its scalar curvature $\kappa_N$ satisfies $\inf\{\kappa_N\}=:\kappa_0>0$.
		Now let $T$ and $A$ denote the O'Neill tensors of the submersion $\pi$ (their definitions can be found in \cite{O'Neill}). By Theorem \ref{thm:eq.Vilms} above, the fibres of $M$ are totally geodesic, so $T=0$. Pick an orthonormal basis of horizontal vector fields $\{X_i\}$. By $G$-invariance, we have that for any point $p\in M$,
		$$\sum_{i,j} ||A_{X_i}(X_j)||_p = \sum_{i,j} ||A_{X_i}(X_j)||_{gp}$$
		for all group elements $g\in G$. This means 
		$\sup_{p\in X}\{\sum_{i,j} ||A_{X_i}(X_j)||_p\} =: A_0<\infty$, as $M/G$ is compact.
		Now by a result of Kramer (\cite{Kramer} p. 596), we can relate the scalar curvatures by
		$$
		\kappa_M(p) = \kappa_{G/K} + \kappa_N(p) - \sum_{i,j} ||A_{X_i}(X_j)||_p.
		$$
		Upon scaling the fibre metric on $N$ by a positive factor $t$, we obtain
		$$
		\kappa_M(p) \ge \kappa_{G/K} + t^{-2}\kappa_0 - A_0 >0 \qquad \text{whenever}\quad 0<t< \sqrt{\frac{\kappa_0}
			{-\kappa_{G/K} + A_0}}.
		$$
		Thus $g_M$ is a $G$-invariant metric of positive scalar curvature on $M$.
	\end{proof}
	
	This enables us to establish the following existence theorem for PSC metrics:
	\begin{theorem}\label{thm:PSCexistence}
		Let $G$ be an almost-connected Lie group acting properly and co-compactly on $M$ and let $K$ be a maximal compact subgroup of $G$ such that the identity component of $K$ is non-abelian. 
		If there is a global slice such that $K$ acts effectively on it,  then  $M$ has a $G$-invariant Riemannian metric of positive scalar curvature.
	\end{theorem}

	\begin{proof} 
		This theorem follows from Theorem \ref{thm:induction} and the following theorem of Lawson and Yau:
		
		\begin{theorem}[\cite{Lawson-Yau}]
			Let $K$ be a compact Lie group whose identity component is nonabelian. If $N$ is a compact manifold with an effective action of $K$,
			then $N$ admits a $K$-invariant Riemannian metric of positive scalar curvature.
		\end{theorem}
	\end{proof}
	
	\section{Two Applications of the Induction Principle}\label{sec:applications}

	The induction principle exhibited in the earlier parts of this paper allows us to generalise a few interesting results involving compact group actions. We owe inspiration to the paper of Hochs-Mathai \cite{HM16},
	where the theorem of Atiyah and Hirzebruch is generalised to the non-compact setting. 
	
	\subsection{Hattori's Vanishing Theorem}
	
	In 1978, Hattori (\cite{Hat}, Theorem 1 and Lemma 3.1) proved the following interesting result: 
	\begin{theorem}\label{thm:Hat0}
		Let $Y$ be a compact, connected, almost-complex manifold of dimension greater than $2$ on which $S^1$ 
		acts smoothly and non-trivially, preserving the almost-complex structure. Suppose that 
		the first Betti number of $Y$ vanishes and that the first Chern class is
		$$
		c_1(Y) = k_0 x, 
		$$
		where $k_0 \in \N$ and $x\in H^2(Y, \Z)$. Let $L$ be a line bundle with $c_1(L) = k x$, for an integer $k$ satisfying $|k|<k_0$ and $k=k_0$ (mod $2$). Then 
		$$
		\ind_{S^1}(\partial^L_Y) = 0 \quad \in R(S^1),
		$$
		where $\partial^L_Y$ is the equivariant $\Spinc$-Dirac operator on $Y$ twisted by $L$.
	\end{theorem}

	Hattori's result was inspired by the  vanishing theorem of Atiyah and Hirzebruch \cite{AH} for non-trivial circle actions on compact Spin-manifolds. 
	We first mildly generalise Theorem \ref{thm:Hat0} from non-trivial circle actions to non-trivial actions of compact, connected Lie groups.
	
	\begin{theorem}\label{thm:Hat}
		Let $Y$ be a compact, connected, almost-complex manifold of dimension greater than $2$ and $K$ a compact,
		connected Lie group acting smoothly and non-trivially on $Y$, preserving the almost-complex structure. Suppose also that 
		the first Betti number of $Y$ vanishes and that the first Chern class is
		$$
		c_1(Y) = k_0 x, 
		$$
		where $k_0 \in \N$ and $x\in H^2(Y, \Z)$. Let $L$ be a line bundle with $c_1(L) = k x$, for an integer $k$ satisfying $|k|<k_0$ and $k=k_0$ (mod $2$). Then 
		$$
		\ind_{K}(\partial^L_Y) = 0 \quad \in R(K),
		$$
		where $\partial^L_Y$ is the equivariant $\Spinc$-Dirac operator on $Y$ twisted by $L$.
	\end{theorem}
	
	\begin{proof}
		For any $g\in K$, there is a maximal torus $T$ of $K$ containing $g$, whence
		$$
		\ind_K(\partial^L_Y)(g) = \ind_{T}(\partial^L_Y)(g).
		$$
		Since by assumption $K$ is connected and acts on $Y$ non-trivially, the action $T$ of on $Y$ given by restriction must also be non-trivial. In the special case $T = S^1\times S^1$, this means that at most two circles in $T$ can act trivially on $Y$. Thus the circles in $T$ that act non-trivially on $Y$ form a dense subset of $T$. For each $t$ in such a circle $S < T$, we have that
		$$
		\ind_T(\partial^L_Y)(t) = \ind_{S}(\partial^L_Y)(t) = 0
		$$
		by Hattori's Theorem \ref{thm:Hat0}. Since the character of a $K$-representation is continuous in $K$, $\ind_T(\partial^L_Y)(t) = 0$ for all $t\in T$. This completes the argument for $T = S^1\times S^1$. 
		
		The argument extends to the general case by induction on the dimension of $T$. One can show that for an $l$-dimensional torus $T$ acting non-trivially on $Y$, no more than $l-1$ circles can act trivially. From this it follows that, under the hypotheses of the theorem, $\ind_K(\partial^L_Y) = 0 \in R(K).$ 
	\end{proof}
	
	Our goal in the following subsection is to extend Theorem  \ref{thm:Hat} to the non-compact setting.
	The result is Theorem \ref{thm:H-M}, which can be stated equivalently in the form of Theorem \ref{thm:H-M2}. 
	Let $X$ be a 
	manifold on which a connected Lie group $G$ acts properly and isometrically. Suppose that the action is cocompact and that $X$ has a $G$-equivariant 
	$\Spin^c$-structure. Let 
	\[
	\ind_G(\partial_X^L)  \in K_{\bullet}(C^*_r(G))
	\]
	be the equivariant index of the associated $\Spinc$-Dirac operator. 
	
	Let $K<G$ be a maximal compact subgroup, and suppose $G/K$ has an almost-complex structure (this is always true for a double cover of $G$, as pointed out in Remark~\ref{rem:G/Kspinc}). 
	
	\begin{definition}[\cite{HM16}]
		The action of $G$ on $X$ is \emph{properly trivial} if all stabilisers are maximal compact subgroups of $G$. For a proper action, the stabilisers cannot be larger. The action is called \emph{properly non-trivial} if it is not properly trivial.
	\end{definition}

	Hattori's Theorem \ref{thm:Hat} generalises as follows.
	\begin{theorem}\label{thm:H-M}
		As above, let $G$ be a connected Lie group with a maximal compact subgroup $K$, and assume that $G/K$ has an almost-complex structure. Suppose that $G$ acts properly and cocompactly on a connected almost-complex manifold $X$ and that  $G$ preserves the almost-complex structure. Suppose also that 
		the first Betti number of $X$ vanishes and that the first Chern class is
		$$
		c_1(X) = k_0 x, 
		$$
		where $k_0 \in \N$ and $x\in H^2(X, \Z)$.  Assume that the $G$-action on $X$ is properly non-trivial. Let $L$ be a $G$-equivariant line bundle with $c_1(L) = k x$, for an integer $k$ satisfying $|k|<k_0$ and $k=k_0$ (mod $2$). Then 
		$$
		\ind_G(\partial_X^L) = 0 \quad \in K_\bullet(C^*_r(G)),
		$$
		where $\partial^L_X$ is the equivariant $\Spinc$-Dirac operator on $X$ twisted by $L$.
	\end{theorem}
	
	Theorem~\ref{thm:index.comm.induction} (\emph{higher index commutes with induction})
	allows us to deduce Theorem~\ref{thm:H-M}
	from Hattori's Theorem \ref{thm:Hat}. 
	This is based on the fact that the Dirac induction map \eqref{eq:indG/K1} relates the equivariant indices of the $\Spinc$-Dirac operators $\partial_Y^L$ on $Y$ and $\partial_X^L$ on $X$, associated to the $\Spinc$-structures $P_Y$ and $P_X$ respectively, to one another. (See Corollary~\ref{prop quant ind}, also~{\cite[Theorem 5.7]{HochsMathaiAJM}}.)
	\medskip
	
	\noindent \emph{Proof of Theorem \ref{thm:H-M}.}
	Let $Y\subset X$ be as in Abels' Theorem (Theorem~\ref{thm:abels}). 
	Using the decomposition~(\ref{eq:decom.tangent}) of the tangent bundle of $X$, the almost-complex structures on $X$ and on $G/K$ (hence on $\mathfrak p$) give rise to an almost-complex structure on $G\times_K TY$. Since $G$ preserves the almost-complex structure, we obtain an almost-complex structure on $Y$ preserved by the $K$-action. 
	Note that a manifold with an almost-complex structure preserved under a group action has an equivariant $\Spinc$-structure.
	By Corollary~\ref{prop quant ind}, we have
	\begin{equation} \label{eq quant ind pf}
	\ind_G(\partial_X^L) = \DInd_K^G\bigl( \ind_K(\partial_Y^{L|_Y})\bigr).
	\end{equation}
	Let $X_{(K)}$ be the set of points in $X$ with stabilisers conjugate to $K$.
	By Lemma~9 of~\cite{HM16}, the fixed point set $Y^K$ of the action by $K$ on $Y$ is related to the action of $G$ on $X$ by
	\[
	X_{(K)} = G \cdot Y^K \cong G/K \times Y^K.
	\]
	The stabiliser of a point $m \in X$ is a maximal compact subgroup of $G$ if and only if $m \in X_{(K)}$. 
	Thus, the condition on the stabilisers of the action of $G$ on $X$ is equivalent to the requirement that the action of $K$ on $Y$ be non-trivial. 
	Moreover, because $G/K$ is contractible, we have 
	$$H^{*}(G\times_K Y, \Z)\simeq H^*(Y,\mathbb{Z}),$$ 
	while $T(G/K)$ is trivial. Hence, $X=G\times_K Y$ and $Y$ have the same first Betti number. In a similar fashion, we have 
	$$c_1(G\times_K TY)=c_1(TY),\qquad c_1(L)=c_1(G\times_K L|_{Y})=c_1(L|_Y).$$ 
	Theorem~\ref{thm:Hat} now implies that, under our hypotheses,
	\[
	\ind_K(\partial_Y^{L|_Y}) = 0.
	\]
	The theorem then follows from relation (\ref{eq quant ind pf}). 
	\hfill $\square$
	\medskip
	%
	%
	
	The Theorem we have just proved has the following equivalent statement.
	\begin{theorem}\label{thm:H-M2}
		In the setting of Theorem \ref{thm:H-M}, $\ind_G(\partial^L_X) \not= 0$ if and only if there is a compact $\Spin^c$-manifold $Y$ with $e^{kx/2} \hat A(Y) \not=0$, and a $G$-equivariant diffeomorphism
		\[
		X \simeq G/K \times Y,
		\]
		where $G$ acts trivially on $Y$.
	\end{theorem}
	
	\begin{proof}
		Because the Dirac induction map $\DInd_K^G$  in \eqref{eq quant ind pf} is an isomorphism,
		\[
		\ind_G(\partial_X^L) \not=0 \quad \Leftrightarrow \quad \ind_K(\partial_Y^{L|_Y}) \not=0. 
		\]
		Moreover, we have an equivalence
		\[
		\ind_K(\partial_Y^{L|_Y}) \not=0 \quad \Leftrightarrow \quad \text{$K$ acts trivially on $Y$ and $e^{kx}\hat A(Y)\not=0$},
		\]
		which follows from Theorem  \ref{thm:Hat}, since if $K$ acts trivially on $Y$, then 
		\mbox{$\ind_K(\partial_Y^{L|_Y})\in R(K)$} equals $\ind(\partial_Y^{L|_Y}) =e^{kx} \hat A(Y)$ copies of the trivial representation. Since $K$ acts trivially on $Y$ if and only if $X = (G/K) \times Y$, the claim follows.
	\end{proof}
	
	\begin{remark}
		The proofs of Theorems \ref{thm:H-M} and \ref{thm:H-M2} are inspired by the proofs of Theorems~2 and 3 in~\cite{HM16}. 
		Moreover (see also Remark~7 in \cite{HM16}), the non-vanishing of $\ind_G(\partial_X^L)$ in Theorems \ref{thm:H-M} and \ref{thm:H-M2} can be replaced by the non-vanishing of the class $p_*[\partial_X^L]$ in~(\ref{eq:ind.X.G/K}) because of the Baum-Connes isomorphism~ (\ref{eq:iso.G/K.BC}).
	\end{remark}

	
	\subsection{An Analogue of Petrie's Conjecture}
	
	The Pontryagin class of a closed oriented manifold is not usually a homotopy invariant. However, for the K\"ahler manifolds $\C P^n$, Petrie~\cite{Petrie} has an interesting conjecture, motivated by the question of whether or not manifolds of a given homotopy type admit non-trivial circle actions. 
	Recall that the total Pontryagin class of $\C P^n$ is 
	\[
	p(\C P^n)=(1+x^2)^{n+1}, \qquad x\in H^2(\C P^n).
	\] 
	There is a natural action of $S^1$ on $\C P^n$ given by 
	\[
	(\lambda,[z_0:\ldots:z_n])\mapsto [\lambda^{a_0} z_0:\ldots:\lambda^{a_n}z_n],\qquad  a_i\in\mathbb{Z}.
	\]
	
	\begin{conjecture}[Petrie's Conjecture~\cite{Petrie}]
		\label{conj:Petrie}
		If a closed oriented manifold $Y$ of dimension $2n$ admits an orientation-preserving homotopy
		equivalence $f: Y\rightarrow \C P^n$ and also a non-trivial circle action, then its total Pontryagin class is given by
		\[
		p(Y)=f^*p(\C P^n).
		\]
	\end{conjecture}
	
	It is a theorem of Hattori~\cite{Hat} that this conjecture holds if $Y$ has an almost-complex structure preserved under the $S^1$-action and if $c_1(Y)=\pm(n+1)x$, where $x\in H^2(Y,\Z)$ is the generator. 
	
	Recall that the $\cL$-class is given by the multiplicative sequence of polynomials belonging to the power series $\frac{\sqrt{t}}{\tanh\sqrt{t}}$ and that the signature of a smooth compact oriented manifold $Y^{4k}$ is equal to the $\cL$-genus $\cL[Y^{4k}].$ We have the following lemma, inspired by a result of Kaminker ~\cite{Kaminker}: 
	
	\begin{lemma}
		\label{lem:h.e.CPn}
		Let $Y$ be a compact $\Spinc$-manifold.
		Let $D_Y, D_{\C P^n}$ denote signature operators on $Y$ and $\C P^n$ respectively. Then
		$\cL(Y)=f^*\cL(\C P^n)$ if and only if
		\[ 
		\quad f_* [D_Y]=[D_{\C P^n}] \in K^{geo}_{\bullet}(\C P^n).
		\]
	\end{lemma}
	
	\begin{proof}
		Assume without loss of generality that $\dim Y$ is even. Poincar\'{e} duality for geometric $K$-homology is given by taking cap product with the fundamental class $[Y]$ (cf. Example~\ref{ex:fundamental.class} and Corollary~\ref{cor:spinc}):
		\[
		\mathcal P\mathcal D_K: K^0(Y)\rightarrow K_0^{geo}(Y)\quad\quad [E]\rightarrow [Y]\cap [E]:=[(Y,E\otimes\C_Y,\id_Y)].
		\]
		Recall that there is a Chern character for K-homology,
		\[
		\ch: K_0(Y)\rightarrow H_{\mathrm{ev}}(Y)\otimes\Q\quad\quad [(Y,E,f)]\mapsto f_*(\mathcal P\mathcal D_H(\ch(E)\cup\mathrm{td}(Y))),
		\]
		that makes the following diagram commutative:
		\begin{equation}
		\label{eq:comm.K.H}
		\begin{CD}
		K^0(Y) @> \ch(-)\cup\mathrm{td}(Y)>\simeq > H^{\mathrm{ev}}(Y)\otimes\Q \\
		@V \mathcal P\mathcal D_K V\simeq V  @V\simeq V \mathcal P\mathcal D_H V \\
		K_0(Y)   @>>\mathrm{ch(-)}> H_{\mathrm{ev}}(Y)\otimes\Q.
		\end{CD}
		\end{equation}
		In this diagram, $\ch(-)\cup\mathrm{td}(Y)$ is the Atiyah-Singer integrand and $\mathcal P\mathcal D_H$ is the Poincar\'e duality map on homology and cohomology.
		Note that if $D_Y$ is the signature operator on $Y$, then $\ch[\sigma(D_Y)]\cup\mathrm{td}(Y)=\mathcal{L}(Y)$. 
		Together with the commutative diagram~(\ref{eq:comm.K.H}) we have 
		\[
		\mathcal P\mathcal D_H(\cL(Y)) = \cL(Y) \cap [Y]=\ch(\mathcal P\mathcal D_K([\sigma(D_Y)])).
		\]
		The proof is then completed by naturality of the Chern character (note that $\cL(Y)=f^*\cL(\C P^n)$ if and only if the intersection products are equal, if and only if
		$f_*(\cL(Y) \cap [Y])=\cL(\C P^n) \cap [\C P^n]$).
	\end{proof}
	
	Using induction on $K$-homology, we are able to generalise a weaker version of Hattori's result on Petrie's Conjecture, as follows.
	
	\begin{theorem}[Analogue of Petrie's Conjecture]
		\label{thm:petrie}
		Let $X$ be a connected manifold admitting an almost-complex structure and a properly non-trivial $SU(1, 1)$-action preserving the almost-complex structure.
		If there is an $SU(1, 1)$-equivariant homotopy equivalence 
		\[
		f: X\rightarrow SU(1,1)\times_{U(1)}\C P^n,
		\]
		and $c_1(X)=\pm(n+1)x$ with $x\in H^2(X,\mathbb{Z})$ being the generator,
		then
		\[
		f_*[D_X]=[D_{SU(1, 1)\times_{U(1)}\C P^n}]\in K^{SU(1,1)}_{\bullet}(SU(1,1)\times_{U(1)}\C P^n) \simeq K^{U(1)}_{\bullet}(\C P^n).
		\]
	\end{theorem}
	
	\begin{proof}
		Let $Y$ be a $U(1)$-slice of $X$.
		There exist $SU(1,1)$-equivariant maps 
		\[
		f: SU(1, 1)\times_{U(1)} Y\rightarrow SU(1,1)\times_{U(1)}\C P^n,\qquad g: SU(1,1)\times_{U(1)}\C P^n\rightarrow SU(1, 1)\times_{U(1)} Y,
		\]
		such that $g\circ f$ and $f\circ g$ are $SU(1,1)$-equivariantly homotopic to the respective identity maps. 
		These maps induce $U(1)$-equivariant homotopies from $f|_Y\circ g|_{\C P^n}$ and $g|_{\C P^n}\circ f|_Y$ to the identity maps.
		Thus, $Y$ is $U(1)$-homotopic to $\C P^n$. 
		Hattori's result now implies that $\cL(Y)=f^*\cL(\C P^n).$ 
		Finally, by Lemma~\ref{lem:h.e.CPn} and induction on $K$-homology, the conclusion follows.
	\end{proof}
	\appendix
	\section{Equality of Mathai-Zhang Index and Integration Trace}
	\label{sec:AppA}
		Let $M$ be a smooth manifold on which a locally compact group $G$ acts cocompactly. Let $dg$ be a fixed left-invariant Haar measure on $G$ and $\chi$ be the modular character of $G$, defined by $dg^{-1}=\chi(g)dg$. Suppose that $M$ has a $G$-invariant Riemannian metric and that $D$ is a $G$-invariant Dirac operator on $M$ acting on sections of a $G$-equivariant Dirac bundle $E\rightarrow M$. Then $D$ defines an element of the $G$-equivariant analytic $K$-homology $K^G_0(M)$,
		$$[D]=\left[\left(L^2(E),\phi,\frac{D}{\sqrt{D^2+1}}\right)\right]\in K^G_0(M),$$
		where $\phi:C_0(M)\rightarrow L^2(E)$ is the usual representation given by pointwise multiplication.
		
		It was shown in Bunke's appendix to \cite{MZ} that, for $G$ a {\it unimodular} locally compact group, the Mathai-Zhang index of $D$ (see \cite{MZ} section 2 for the definitions) equals the composition
		$$I:K_0^G(M)\xrightarrow{\ind_G} K_0(C^*(G))\xrightarrow{\int_G}\mathbb{Z}$$
		applied to $[D]$. Here $\ind_G$ is the equivariant (analytic) index map and $\int_G$ is the integration trace on $K_0(C^*(G))$. More precisely, $\int_G$ is induced by the map $C^*(G)\rightarrow\mathbb{Z}$ defined on the dense subalgebra $C_c(G)$ by 
		$$f\mapsto\int_G f(g)\,dg.$$
		
		The purpose of this appendix is to provide a proof, using Bunke's technique, that the Mathai-Zhang index equals the integration trace when $G$ is an arbitrary locally compact group.
				
		Let $c$ be a (compactly supported) cut-off function on $M$, and let $L^2_c(E)^G$ be the $L^2$-completion of the space $\{cs:s\in C(E)^G\}$, where $C(E)^G$ is the space of $G$-invariant continuous sections of $E$. We define the Sobolev analogues $H^i_c(E)^G$, $i\geq 0$, analogously. One sees that elements of $H^i_c(E)^G$ are precisely those of the form $c\mu$, where $\mu\in H^i_{\textnormal{loc}}(E)^G$, the space of $G$-invariant, locally $H^i$ sections.
		
		Let $P_c^\chi: L^2(E)\rightarrow L^2_c(M,E)^G$ be given by
		$$P_c^\chi\mu=\frac{c(x)}{\int_G\chi^-1(g)(c(g^{-1}x)^2)\,dg}\int_G\chi^{-1}(g)c(g^{-1}x)g\mu(g^{-1}x)\,dg.$$ 
		By Proposition 3.1 of \cite{TYZ} (adapted to a left Haar measure), $P_c^\chi$ is the orthogonal projection $L^2(E)\rightarrow L^2_c(E)^G$. Now for each $t\in [0,1]$, define a linear map $P_c^{\chi^t}$ on $L^2(E)\ni\mu$ by		$$P_c^{\chi^t}\mu=\frac{c(x)}{\int_G(\chi^t)^{-1}(g)(c(g^{-1}x)^2)\,dg}\int_G(\chi^t)^{-1}(g)c(g^{-1}x)g\mu(g^{-1}x)\,dg.$$
		One can verify that each $P_c^{\chi^t}$ is a projection $L^2(E)\rightarrow L^2_c(E)^G$, orthogonal when $t=1$.
		
		For each $t$, $P_c^{\chi^t}$ is a bounded operator, and the path $t\mapsto P_c^{\chi^t}$ is continuous in the strong$^*$ operator topology on $L^2(E)$. Now for each $t\in [0,1]$, define $\tilde{D}_t:H^1_c(E)^G\rightarrow L^2_c(E)^G$ by 
		$$c\mu\mapsto P^{\chi^t}_c D(c\mu),$$
		for $\mu\in H^1_\textnormal{loc}(E)^G.$ It follows that the path $t\mapsto\tilde{D}_t$ is continuous in the same sense and has end points
		$$\tilde{D}_0: c\mu\mapsto P^1_c D(c\mu),\qquad\tilde{D}_1: c\mu\mapsto P^\chi_c D(c\mu).$$
		A variant of the method used to prove Lemma D.1 in \cite{MZ} shows that each $\tilde{D}_t$ is Fredholm.
		\begin{proposition}
		In the above setting, the Mathai-Zhang index of $D$ is equal to $I([D])$.
		\end{proposition}
		\begin{proof}
		The path $t\mapsto \tilde{D}_t$ of Fredholm operators defined above gives an equality of $K$-homology classes 
		$$[\tilde{D}_1] = [\tilde{D}_0]\in K_0^G(M),$$
		where for each $t\in [0,1]$, we define the class $[\tilde{D}_t]$ in the obvious way, namely
		$$[\tilde{D}_t]:=\left[\left(L^2(E),\phi,\frac{\tilde{D}_t}{\sqrt{\tilde{D}_t^2+1}}\right)\right]\in K^G_0(M).$$
		The appendix of \cite{MZ} shows that
		$$\ind(\tilde{D}_0) = I([D]),$$
		where on the left is the Fredholm index of $\tilde{D}_0$. On the other hand, the Mathai-Zhang index of the Dirac operator $D$ is equal to $\ind(\tilde{D}_1)$ by definition, and we conclude.
		\end{proof}



\begin{thebibliography}{99}
		
		\bibitem{Abels} {H.\ Abels}, {\it Parallelizability of proper actions, global K-slices and maximal compact subgroups}, 
		Math.\ Ann.\ 212 (1974), 1--19.
		
		\bibitem{AK} Y. Arano, Y. Kubota, {\it A categorical perspective on the Atiyah-Segal completion theorem in $KK$-theory} ArXiv: 1508.06815.
		
		\bibitem{AH}
		M.~Atiyah, F.~Hirzebruch,  {\it Spin-manifolds and group actions},  1970 Essays on Topology and Related Topics (M\'emoires d\'edi\'es � Georges de Rham) Springer, New York, 18--28.
		
		\bibitem{ASe} M.~Atiyah, G.B. Segal, {\it Equivariant $K$-theory and completion}, J. Differential Geometry 3 (1969), 1--18.
		
		\bibitem{AS} M. F. Atiyah, I. M. Singer, {\it The index of elliptic operators. I.} Ann. of Math. (2) 87 1968 484--530. 
		
		\bibitem{Atiyah70} M. F. Atiyah, {\it Global theory of elliptic operators.} In Proc. of the International Symposium on Functional Analysis, Tokyo, 1970. University of Tokyo Press. 1.
		
		
		\bibitem{AtiyahSchmid} M. F. Atiyah, W. Schmid, {\it A geometric construction of the discrete series for semisimple
			{L}ie groups} Inventiones Mathematicae 42 1977 1--62.
		
		\bibitem{Baum-Connes} {P. Baum, A. Connes}, {\it Geometric {$K$}-theory for {L}ie groups and foliations}, {Enseign. Math. (2)},
		Vol {46}, (2000), No. {1-2}, {3--42}.
		
		\bibitem{BCH} P. Baum, A. Connes, N. Higson, 
		{\it Classifying space for proper actions and K-theory of group C$^*$-algebras},  C$^*$-algebras: 1943-1993 (San Antonio, TX, 1993), Contemp. Math., 167, Amer. Math. Soc., Providence, RI (1994), 240--291. 
		
		\bibitem{Baum-Douglas82} P. Baum, R. Douglas, {\it K-homology and index theory.} In R. Kadison, editor, Operator Algebras and Applications, volume 38 of Proceedings of Symposia in Pure Mathematics, pages 117--173, Providence, RI, 1982. American Mathematical Society.
		
		\bibitem{BaumHigsonSchick.eq} P. Baum, N. Higson, T. Schick, {\it A geometric description of equivariant K-homology for proper actions.} Quanta of maths, 1-22, Clay Math. Proc., 11, Amer. Math. Soc., Providence, RI, 2010.
		
		\bibitem{BaumHigsonSchick} P. Baum, N. Higson, T. Schick, {\it  On the equivalence of geometric and analytic K-homology.} Pure Appl. Math. Q. 3 (2007), no. 1, part 3, 1--24.
		
		\bibitem{BOSW} P. Baum, H. Oyono-Oyono, T. Schick, M. Walter, 
		{\it Equivariant geometric K-homology for compact Lie group actions. }
		Abh. Math. Semin. Univ. Hambg. 80 (2010), no. 2, 149--173. 
		
		\bibitem{Bismut} J.M. Bismut, {\it The Atiyah-Singer index theorem for families of Dirac operators: two heat equation proofs.} Invent. Math. 83 (1986), no. 1, 91--151.
		
		\bibitem{Bott-Tu} R. Bott and L. Tu,
		Differential forms in algebraic topology, GTM {\bf 82}.
		
		\bibitem{BDF} L. Brown, R. Douglas, and P. Fillmore, {\it Extensions of $C^*$-algebras and K-homology.} Annals of Math., 105: 265--324, 1977.
		

		\bibitem{Blackadar} B. Blackadar, {\it {$K$}-theory for operator algebras (2nd ed.).} Mathematical Sciences Research Institute Publications, Cambridge University Press, Cambridge (1998) xx+300.

		\bibitem{Cecchini} {S. Cecchini}, {\it Callias-type operators in C∗-algebras and positive scalar curvature on noncompact manifolds.} Preprint, https://arxiv.org/abs/1611.01800.
		
		\bibitem{CEN} {J.\ Chabert, S.\ Echterhoff  and R.\ Nest}, {\it The Connes--Kasparov conjecture for almost-connected groups and for linear $p$-adic groups},  Publ.\ Math.\ Inst.\ Hautes \'Etudes Sci.  97 (2003), 239--278.
		
		\bibitem{Connes} {A. Connes}, {\it An analogue of the {T}hom isomorphism for crossed products of a {$C^{\ast} $}-algebra by an action of {${\bf R}$}}, {Adv. in Math.},
		Vol {39}, (1981), No. {1}, 31--55.
		
		\bibitem{Connes-Moscovici} A. Connes and H. Moscovici, {\it Cyclic cohomology, the Novikov conjecture and hyperbolic groups.} Topology Vol. 29 No. 3 345-388, (1990).
		
		\bibitem{CM82} A. Connes and H. Moscovici, \emph{$L^2$-index theorem for homogeneous spaces of Lie groups.} Ann. Math. Vol. {\bf 115}, No. 2, 291--330 (1982).
		
		\bibitem{Ebert} {J. Ebert}, {\it Index theory in spaces of noncompact manifolds I: Analytical foundations.} Preprint, {https://arxiv.org/abs/1608.01699}.
		
		\bibitem{EE} S. Echterhoff, H. Emerson, {\it Structure and $K$-theory of crossed products by proper actions.} Expo. Math. 29 (2011), no. 3, 300--344.
		
		\bibitem{EmersonMeyer1} H. Emerson, R. Meyer, {\it Bivariant {$K$}-theory via correspondences.} Adv. Math. 225 (2010), no. 5, 2883--2919.
			
		\bibitem{EmersonMeyer2} H. Emerson, R. Meyer, {\it Equivariant {L}efschetz maps for simplicial complexes and smooth manifolds.} Math. Ann. 345 (2009), no. 3, 599--630.
		
		\bibitem{Raeburn} J. Fox, P. Haskell, I. Raeburn, {\it Kasparov products, $KK$-equivalence, and proper actions of connected reductive Lie groups}, J. Operator Theory 22 (1989), 3--29.
		
		\bibitem{Fukumoto} Y. Fukumoto, {\it On the Strong Novikov Conjecture of Locally Compact Groups for Low Degree Cohomology Classes.} Preprint, https://arxiv.org/abs/1604.00464.
		
		\bibitem{Gromov-Lawson} M. Gromov and H. B. Lawson, Positive Scalar Curvature and the Dirac Operator on Complete Riemannian Manifolds. 83-196. Publications Mathématiques de l'Institut des Hautes Scientifiques, 58, 1, (1983). 
		
		\bibitem{Hat} 
		A. Hattori, {\it $Spin^c$-structures and $S^1$-actions}, Inven. Math. {\bf 48} (1978) 7-31.
		
		\bibitem{Hermann}
		R. Hermann, {\it A sufficient condition that a mapping of {R}iemannian manifolds be a fibre bundle}, Proc. Amer. Math. Soc., {\bf 11} (1960) 236--242.
		
		\bibitem{Higson-Roe} N. Higson and J. Roe, 
		Analytic K-homology. Oxford Mathematical Monographs. Oxford Science Publications. Oxford University Press, Oxford, 2000. xviii+405 pp.
		
		\bibitem{HochsDS} 
		P.\ Hochs, {\it Quantisation commutes with reduction at discrete series representations of semisimple groups}, Adv.\ Math.\ 222 (2009), no.\ 3, 862--919.
		
		\bibitem{HochsPS} 
		P.\ Hochs, 
		{\it Quantisation of presymplectic manifolds, $K$-theory and group representations}, 
		Proc.\ Amer.\ Math.\ Soc.\ 143 (2015), no. 6, 2675-2692.
		
		\bibitem{HMAIM} P.\  Hochs, V.\ Mathai, {\it Geometric quantization and families of inner products}, 
		{Adv. Math.}, 
		\textbf{ 282} (2015) 362-426. ArXiv:1309.6760
		
		\bibitem{HochsMathaiAJM} P.\  Hochs, V.\ Mathai, {\it Quantising proper actions on $\Spinc$-manifolds}, {Asian J. Math.} (to appear) 62 pages, ArXiv:1408.0085.
		
		\bibitem{HM16} P.\  Hochs, V.\ Mathai, {\it $\Spin$-structures and proper group actions}, 
		{Adv. Math.}, 
		\textbf{ 292} (2016) 1-10. ArXiv:1411.0781.
		
		\bibitem{Julg} 
		P. Julg, 
		{\em K-th\'eorie \'equivariante et produits crois\'es}. (French. English summary) [Equivariant K-theory and crossed products] 
		C. R. Acad. Sci. Paris S\'er. I Math. 292 (1981), no. 13, 629--632. 
		
		\bibitem{Kaminker} J. Kaminker, J.G. Miller, \emph{Homotopy invariance of the analytic index of signature operators over $C^*$-algebras.} J. Operator Theory 14 (1985), no. 1, 113--127.
		
		\bibitem{Kankaanrinta} M. Kankaanrinta, \emph{Equivariant collaring, tubular neighbourhood and gluing
			theorems for proper Lie group actions.} Algebr. Geom. Topol. 7 (2007), 1--27.

		\bibitem{Kasparov2016} G. G. Kasparov, \emph{\it Elliptic and transversally elliptic index theory from the viewpoint of $KK$-theory}. J. Noncommut. Geom. 10 (2016) 4, 1303--1378. 
		
		\bibitem{Kasparov} G. G. Kasparov, {\it Equivariant KK-theory and the Novikov conjecture}. Invent. Math. {\bf 91} (1988), no. 1, 147-201.
		
		
		\bibitem{Kasparov84} {G. G. Kasparov}, {\it Operator {$K$}-theory and its applications: elliptic operators, group representations, higher signatures, {$C^\ast$}-extensions}, {Proceedings of the {I}nternational {C}ongress of {M}athematicians, {V}ol.\ 1, 2 ({W}arsaw, 1983)}, {987--1000},
		(1984).
		
		\bibitem{Kasparov83} G. G. Kasparov, {\it The index of invariant elliptic operators, $K$-theory, and
			Lie group representations.} Dokl. Akad. Nauk SSSR, 268 (1983), 533--537;
		translation: Soviet Mathematics-Doklady, 27 (1983), 105--109.
		
		\bibitem{Kasparov.K-homology} G. G. Kasparov, {\it Topological invariants of elliptic operators I: K- homology.} Math. USSR Izvestija, 9: 751--792, 1975.
		
		\bibitem{KasparovSkandalis} G. G. Kasparov, G. Skandalis, {\it Groups acting properly on ``bolic" spaces and the Novikov conjecture.} Ann. of Math. (2) 158 (2003), no. 1, 165–-206.
		
		\bibitem{Kramer} W. Kramer, The scalar curvature on totally geodesic fiberings. Ann. Global Anal. Geom. 18 (2000), no. 6, 589-600.
		
		\bibitem{LafforgueICM} {V. Lafforgue}, {\it Banach {$KK$}-theory and the {B}aum-{C}onnes conjecture}, {Banach {$KK$}-theory and the {B}aum-{C}onnes conjecture},
		{Proceedings of the {I}nternational {C}ongress of
			{M}athematicians, {V}ol. {II} ({B}eijing, 2002)}, {795--812}
		Vol {149}, (2002), No. {1}, {1--95}.
		
		\bibitem{Lafforgue} {V. Lafforgue}, {\it {$K$}-th\'eorie bivariante pour les alg\`ebres de {B}anach et conjecture de {B}aum-{C}onnes}, {Invent. Math.},
		Vol {149}, (2002), No. {1}, {1--95}.
		
		\bibitem{Lance} {E. C. Lance}, {\it Hilbert {$C^*$}-modules, a toolkit for operator algebraists}, {Cambridge University Press, Cambridge}, {1995}, {x+130}.
		
		\bibitem{Lawson-Yau}  B. Lawson and S. T. Yau, Scalar curvature, non-abelian group actions, and the degree of symmetry of exotic spheres. Comment. Math. Helv. 49 (1974), 232-244.
		
		\bibitem{Lichnerowicz} A. Lichnerowicz, Spineurs harmoniques, C. R. Acad. Sci. Paris. S\'erieA,257 (1963),7-9.
		
		\bibitem{LO}	W. L\"uck and B. Oliver, The completion theorem in K-theory for proper actions of a discrete group. Topology 40 (2001), no. 3, 585-616. 

		\bibitem{Milnor} J. W. Milnor {\it Topology from the differentiable viewpoint, based on notes by David W. Weaver.} {The University Press of Virginia, Charlottesville, Va.} (1965) {ix+65}.

		\bibitem{MZ}
		V.\ Mathai and W.\ Zhang,  {\it Geometric quantization for proper actions} (with an appendix by U.\  Bunke),  Adv.\ Math.\ 225 (2010), no. 3, 1224-1247. ArXiv:0806.3138.
		
		\bibitem{O'Neill}  B. O'Neill,  The fundamental equations of a submersion. Michigan Math. J. 13 (1966) 459-469. 
		
		\bibitem{Palais} R. Palais, On the existence of slices for actions of non-compact {L}ie
		groups. Ann. of Math. (2) 73 (1966) 295-323.
		
		\bibitem{Petrie} T. Petrie, \emph{Smooth $S^1$ actions on homotopy complex projective spaces and related topics.} Bull. Amer. Math. Soc. 78 (1972), 105--153.
		
		\bibitem{Phillips2}  C. N. Phillips, {\it Equivariant K-theory for proper actions. II. Some cases in which finite-dimensional bundles suffice}. Index theory of elliptic operators, foliations, and operator algebras, 205-227, Contemp. Math., 70, Amer. Math. Soc., Providence, RI, 1988. 
		
		\bibitem{Phillips1}  C. N. Phillips,  {\it Equivariant K-theory for proper actions and C*-algebras}. Index theory of elliptic operators, foliations, and operator algebras, 175-204, Contemp. Math., 70, Amer. Math. Soc., Providence, RI, 1988.
		
		
		\bibitem{Rieffel}
		M.A. Rieffel, 
		{\em Applications of strong Morita equivalence to transformation group $C^*$-algebras}. Operator algebras and applications, Part I (Kingston, Ont., 1980), pp. 299--310, 
		Proc. Sympos. Pure Math., {\bf 38}, Amer. Math. Soc., Providence, R.I. (1982) 
		
		\bibitem{Rosenberg} J. Rosenberg, {\it {$C^\ast$}-algebras, positive scalar curvature, and the
			{N}ovikov conjecture. {III}}, Topology, 25 (1986), 319--336.
		
		\bibitem{Segal} G. Segal, {\it Fredholm complexes}, Quart. J. Math. 21 (1970), 385--402.
		
		%
				
		\bibitem{TYZ} X.\ Tang, Y.\ Yao and W.\ Zhang, {\it Hopf cyclic cohomology and Hodge theory for proper actions}, J. Noncommut. Geom. 7 (2013), 885--905.

		\bibitem{Vilms} J. Vilms, Totally geodesic maps. J. Differential Geometry 4 (1970) 73-79. 
		
		\bibitem{Wang} H. Wang, {\it $L^2$-index formula for proper cocompact group actions}.  J. Noncommut. Geom. 8 (2014),  no.2,  393--432.
		
		\bibitem{Wassermann} {A. Wassermann}, {\it Une d\'emonstration de la conjecture de {C}onnes-{K}asparov pour les groupes de {L}ie lin\'eaires connexes r\'eductifs}. {C. R. Acad. Sci. Paris S\'er. I Math.},
		Vol. {304}, (1987) No. 18, {559--562}.
		
		
		\bibitem{Zhang17} W. Zhang, A note on the Lichnerowicz vanishing theorem for proper actions, J. Noncomm. Geom. 11 (2017) 823
		
	\end{thebibliography}
\end{document}